\documentclass[12pt]{amsart}
\usepackage[a4paper]{geometry}
\usepackage[mathscr]{euscript}
\usepackage[all]{xy}
\usepackage{amssymb}
\usepackage{amsmath}
\usepackage{url}
\def\Op{\mathscr{O}'}
\def\Wt{\widetilde{W}}
\DeclareMathOperator{\Tr}{Tr}
\DeclareMathOperator{\Norm}{N}
\def\charp{{\varkappa_p}}
\def\charptwo{{\varkappa_{p,2}}}
\def\hZZ{\widehat{\ZZ}}
\def\hQQ{\widehat{\QQ}}
\def\hH{\widehat{\H}}
\def\hHx{\hH^\times}
\def\hO{\widehat{\O}}
\def\hOx{\hO^\times}
\def\sA{\mathscr{A}}
\def\sE{\mathscr{E}}

\def\smat#1{\left(\begin{smallmatrix}#1\end{smallmatrix}\right)}

\theoremstyle{plain}
\newtheorem{proposition}{Proposition}[section]
\newtheorem{theorem}[proposition]{Theorem}
\newtheorem*{theoremA}{Theorem A}
\newtheorem{corollary}[proposition]{Corollary}

\newtheorem{lemma}[proposition]{Lemma}

\theoremstyle{definition}
\newtheorem{definition}[proposition]{Definition}
\newtheorem{notation}[proposition]{Notation}
\newtheorem{remark}[proposition]{Remark}

\newenvironment{mainthm}
  {\theoremA[Theorem \ref{thm:main}]}
  {\endtheoremA}

\newenvironment{rankinthm}
  {\theorem[Theorem \ref{thm:rankin}]}
  {\endtheorem}

\numberwithin{table}{section}

\DeclareMathOperator{\Bil}{Bil}

\DeclareMathOperator{\Hom}{Hom}

\DeclareMathOperator{\sign}{sign}

\DeclareMathOperator{\disc}{disc}

\def\new{\mathrm{new}}
\def\old{\mathrm{old}}

\newcommand{\calP}{\mathcal P}
\newcommand{\calC}{\mathcal C}
\newcommand{\calD}{\mathcal D}

\newcommand{\Ot}{{\tilde{\Om}}}
\newcommand{\Om}{{\mathscr{O}}}

\def\sA{\mathscr A}
\def\TT{\mathbb T}
\def\ZZ{\mathbb Z}
\def\NN{\mathbb N}
\def\RR{\mathbb R}
\def\CC{\mathbb C}
\def\QQ{\mathbb Q}

\def\<#1>{{\left\langle{#1}\right\rangle}}
\def\abs#1{{\left|{#1}\right|}}

\def\Z{{\mathbb Z}}             
\def\Q{{\mathbb Q}}             

\def\set#1{{\left\{{\def\st{\;:\;}#1}\right\}}}
\def\num#1{{\#{#1}}}
\def\numset#1{{\num{\set{#1}}}}

\let\kro\dkro

\def\H{{B}}                     
\def\Hh{\widehat{\H}}           
\def\Hx{{\H^\times}}            
\def\Hpx{{\H_p^\times}}         

\def\O{{R}}           
\def\Oh{\widehat{\O}} 
\def\Op{{\O_p}}                 
\def\Ol(#1){{\mathop{\O_l}(\id{#1})}}                 
\def\Or(#1){{\mathop{\O_r}(\id{#1})}}                 
\def\Orp(#1){{\mathop{\O_r}(\idp{#1})}}               
\def\Otern(#1){{\mathop{\O^0_r}(\id{a})}}    

\def\id#1{{\mathfrak{#1}}}      
\def\idp#1{{\id{#1}_p}}         
\def\cls#1{{[\id{#1}]}}         
\def\gen#1{{\mathop{\mathrm{gen}}[\id{#1}]}} 

\def\vec#1{{\mathbf{#1}}}       
\def\quat#1{{{\mathit{#1}}}}             
\def\quatp#1{\quat{#1}_p}     

\DeclareMathOperator{\norm}{{\mathscr N}}
\DeclareMathOperator{\trace}{{\mathrm{Tr}}}
\DeclareMathOperator{\normx}{{\Delta}}

\def\normid#1{{\norm{\id{#1}}}}

\DeclareMathOperator{\I}{{\mathscr I}}
\DeclareMathOperator{\Ix}{{\widetilde{\I}}}

\DeclareMathOperator{\M}{{\mathscr M}}
\DeclareMathOperator{\Mt}{\widetilde{{\mathscr M}}}

\def\T_#1(#2){{\mathop{\mathscr T}\nolimits_{#1}(\id{#2})}}
\def\TO(#1)_#2(#3){{\mathop{\mathscr T}\nolimits^{#1}_{#2}(\id{#3})}}
\def\t{{\mathop{t}\nolimits}}
\def\Hecke{{\phi}}

\def\HeckeRing{\mathbb{T}}

\def\A_#1(#2){{\mathop{\mathscr A}\nolimits_{#1}(\id{#2})}}
\def\Ax_#1(#2){{\mathop{\widetilde{\mathscr A}}\nolimits_{#1}(\id{#2})}}

\def\e#1{\vec{e}_{{#1}}}

\def\localconstant#1(#2){\varepsilon_{#1}({#2})}
\def\sign#1(#2){\epsilon_{#1}({#2})}

\def\charp{{\varkappa_p}}

\begin{document}

\title
  [Shimura correspondence for level $p^2$]
  {Shimura correspondence for level $p^2$ and the central values of L-series II}

\author{Ariel Pacetti}
\thanks{The first author was partially supported by CONICET PIP 2010-2012 GI
        and FonCyT BID-PICT 2010-0681}
\address{Departamento de Matem\'atica, Universidad de Buenos Aires,
         Pabell\'on I, Ciudad Universitaria. C.P:1428, Buenos Aires, Argentina}
\email{apacetti@dm.uba.ar}

\author{Gonzalo Tornar{\'\i}a}
\thanks{The second author was partially supported by ANII FCE 2009/2972}
\address{Centro de Matem\'atica, Facultad de Ciencias, Universidad de la Rep\'ublica\\
         Igu\'a 4225 esq. Mataojo, Montevideo, Uruguay}
\email{tornaria@cmat.edu.uy}

\maketitle

\begin{abstract} 
  Given a Hecke eigenform $f$ of weight $2$ and square-free level $N$,
  by the work of Kohnen, there is a unique weight $3/2$ modular form of
  level $4N$ mapping to $f$ under the Shimura
  correspondence. Furthermore, by the work of Waldspurger the Fourier
  coefficients of such a form are related to the quadratic twists of
  the form $f$. Gross gave a construction of the half integral weight
  form when $N$ is prime, and such construction was later generalized
  to square-free levels. However, in the non-square free case, the
  situation is more complicated since the natural construction is
  vacuous. The problem being that there are too many special points so
  that there is cancellation while trying to encode the information as
  a linear combination of theta series.

  In this paper, we concentrate in the case of level $p^2$, for $p>2$
  a prime number, and show how the set of special points can be split
  into subsets (indexed by bilateral ideals for an order of reduced
  discriminant $p^2$) which gives two weight $3/2$ modular forms
  mapping to $f$ under the Shimura correspondence. Moreover, the
  splitting has a geometric interpretation which allows to prove that
  the forms are indeed a linear combination of theta series associated
  to ternary quadratic forms.

  Once such interpretation is given, we extend the method of Gross-Zagier
  to the case where the level and the discriminant are not prime to
  each other to prove a Gross-type formula in this situation.
\end{abstract}

\section*{Introduction}

The theory of modular forms of half-integral weight was developed by
Shimura in \cite{Shimura}. There he defined a map known as the
``Shimura correspondence'' that associates to a modular form $g$ of
half-integral weight $k/2$ ($k$ odd), level $4N$ and character $\psi$,
such that $g$ is an eigenform for the Hecke operators,
a modular form $f$ of weight $k-1$, level $2N$ and 
character $\psi^2$. In the same work, he noted that the Fourier
coefficients of the half-integral weight modular form have more
information than the Fourier coefficients of the integral weight
modular form and raised the question of the meaning of these Fourier
coefficients. In $1981$, Waldspurger answered the question by relating
the Fourier coefficients of $g$ to the central values of twisted
$L$-series of $f$ (see \cite{Waldspurger}).

In \cite{Gross}, Gross gave an explicit method to
construct, when $f$ has weight $2$, prime level $p$ and trivial character,
a weight
$3/2$ modular form (of level $4p$ and trivial character) mapping
to $f$ via the Shimura correspondence.
A formula for the central values of twists of the $L$-series of $f$ by
imaginary quadratic characters was proved by Gross.
His result was later generalized
(by a different method) to odd \emph{square-free} levels $N$ in \cite{BSP2},
where the authors construct one modular form whose Fourier
coefficients are related to central values of twists of the $L$-series
of $f$ by imaginary quadratic characters with
discriminants satisfying $2^t$ quadratic conditions (where $t$ is the
number of prime factors of $N$).

Generalizing Gross method in a different direction, the case of level
$p^2$, for a prime $p>2$, was studied in
\cite{Pacetti-Tornaria}. In that work, given a modular form of weight
$2$ and level $p^2$, \emph{two} weight $3/2$ modular forms were
constructed, and a ``Gross formula'' was conjectured (see \cite[Conjecture
$2$]{Pacetti-Tornaria}). Some examples as well as an application to
computing central values of twists by \emph{real} quadratic characters
were presented in \cite{Pa-To2} and \cite{ell-data}.
The main purpose of this paper is to explain the nature of such
formula and also prove it. 

Let $f=\sum a(n)\,q^n$ be a cusp
form of weight $2$ and level $N$.
Let $K$ be an imaginary quadratic field of discriminant $D<0$ and let
$\Om_K$ be its ring of integers. We denote $\I(\Om_K)$ the class group
of $\Om_K$.

We define the twisted $L$-series
\[
L(f, D, s) := \sum \frac{a(n)}{n^s}\, \kro{D}{n}.
\]
Note that $L(f, D, s)$ as defined here may not be a primitive
$L$-series when $N$ and $D$ have a common factor.
Our main result, Theorem~\ref{thm:main}, deals with the case
$N=p^2$, for $p>2$ an odd prime, and $D=-pd$ (with $p\nmid d$)
a fundamental discriminant, but Section~\ref{sec:bilateral} and the appendix
on Rankin's method
are more general, including the case where $N$ and $D$ have a common
factor as required for Section~\ref{sec:main}.
Note that the results in Section~\ref{sec:specialpoints} 
assume that $D$ is odd, but the main result in Section~\ref{sec:main}
is proved for odd and even discriminants altogether.

The proof of our formula consists of two parts.  First, we need to
compute the Fourier coefficients of the half-integral weight modular
forms constructed in \cite{Pacetti-Tornaria}, and give an
interpretation of such construction in terms of special points. In
Section~\ref{sec:bilateral}, we give an ad\`elic definition of the special points of
discriminant $D$ and in Proposition
\ref{prop:specialpointscorrespondence} we prove that the ad\`elic
definition coincides with Eichler's original formulation. This
interpretation is crucial to split the special points of discriminant
$D$ into subsets indexed by bilateral ideals, as defined in \eqref{eq:splitting}.

To compute the Fourier coefficients of the weight $3/2$ modular forms
attached to such splitting, we fix a quaternion algebra over $\Q$
ramified at $p$ and $\infty$, and an order $\Ot$ with reduced
discriminant $p^2$. We consider the algebra $\HeckeRing_0$ of Hecke
operators with index prime to $p$ acting on $\M(\Ot)$, a vector space
(over $\RR$) spanned by representatives of $\Ot$-ideal classes. Since
the eigenspaces of $\HeckeRing_0$ do not have multiplicity one (in
general), we need to add some extra operators.

The group of bilateral $\Ot$-ideals modulo $\Q^\times$-equivalence is
a dihedral group of order $2(p+1)$. The norm $p$ bilateral
$\Ot$-ideals can be taken as representatives for the symmetries. Each
bilateral $\Ot$-ideal defines an operator in $\M(\Ot)$. The operator
$W_{\id{p}}$ associated to a norm $p$ bilateral $\Ot$-ideal $\id{p}$
commutes with $\HeckeRing_0$ and is self-adjoint for the natural inner
product of $\M(\Ot)$ (defined in \eqref{eq:heightpairing} below), although
operators related to different norm $p$ bilateral $\Ot$-ideals clearly
do not commute. Considering the algebra generated by the Hecke
operators and one $W_{\id{p}}$, a multiplicity one theorem does hold
(see Theorem \ref{thm:multone}).

We show that there is a connection between bilateral $\Ot$-ideals of
norm $p$ and suborders of $\Ot$ (in the sense of
\cite{Pacetti-Tornaria}) which associates to $\id{p}$ the order $\ZZ +
\id{p}$. This relation allows to define for each ideal $\id{p}$, a map
\[
\Theta_{\id{p}}: \M(\Ot) \mapsto M_{3/2}(4p^2,\charp),
\] 
where $\charp(n) := \kro{p}{n}$ (see Section \ref{thetacorr}). If $\vec{e}$ is an
eigenvector for the Hecke operators such that
$\Theta_{\id{p}}(\vec{e}) \neq 0$, then $W_{\id{p}}(\vec{e}) =
\vec{e}$ (see Proposition~\ref{prop:Wp-action}). This implies that the
only eigenvectors to be considered for $\Theta_{\id{p}}$ are those
where $W_{\id{p}}$ acts trivially. This explains the orthogonality
condition on the eigenvector $\vec{e}_{f,\Ot}$ needed for Conjecture 2
in \cite{Pacetti-Tornaria}.

In Section~\ref{sec:specialpoints} we relate the Fourier coefficients
of a modular form coming from Rankin's method (see Theorem~\ref{thm:rankin})
to the height of special points. This allows,
given a character $\varphi$ of the class group $\I(\Om_K)$, to relate the
central value of $$L_\varphi(f,s):=\sum_{\sA}\varphi(\sA)L_\sA(f,s)$$
to the height of a sum of special points
of discriminant $D$ (Proposition~\ref{prop:centralvalue}).
The special case $\varphi=1_D$ (the trivial character on $\I(\Om_K)$)
relates to central values of twists of
the $L$-series by the factorization
\[
   L_{1_D}(f,s) = L(f, s) \, L(f, D, s).
\]

The second part of the proof is the well known Rankin's method. This
part is a little more technical hence it is left to the appendix.
Following \cite{Gross} and \cite{GZ} we define a Rankin
convolution $L$-series and using Rankin's method we compute its
central value. This is done in a very similar way than
that of \cite{GZ}. The difficulty in our case comes from the fact that
the level of $f$ and the discriminant of the imaginary quadratic field
are not prime to each other. The formula for the central value of
$L_{\sA}(f,s)$ (for an ideal class $\sA\in\I(\Om_K)$) proved in
Theorem \ref{thm:rankin} is similar to \cite[Proposition 4.4]{GZ} with
the condition $\gcd(N,D)=1$ removed.  In addition to giving a more
complete result in general, this lifts an important restriction when
$N$ is not squarefree; for instance, when $N$ is a perfect square,
both sides in the formula in Theorem \ref{thm:rankin} vanish trivially
for $\gcd(N,D)=1$.

In the last section, we relate the heights of special points to the
Fourier coefficients of ternary theta series. The key idea is to split
the set of special points of discriminant $D$ into $p+1$ subsets,
indexed by the norm $p$ bilateral $\Ot$-ideals. Note that depending on
whether $\kro{D/p}{p}$ is a square or not, half of these subsets will
be empty, while the other half will have the same number of elements. By
counting the number of special points in each set, we conclude the
proof of our main result.

\begin{mainthm}
Let $f$ be a new eigenform of weight $2$, level $p^2$ with $p>2$ an odd prime.
Fix a norm $p$ bilateral $\Ot$-ideal $\id{p}$,
and let $\vec{e}_f$ be an eigenvector in
the $f$-isotypical component of
$\M(\Ot)$ such that $W_{\id{p}}(\vec{e}_f) = \vec{e}_f$.
\par
If $d$ is an integer such that $D=-pd<0$ is a fundamental discriminant, 
and such that $\kro{d}{p}=\chi(\id{p})$, then
\[
L(f,1)\,L(f,D,1)
=
4 \pi^2\,\frac{\<f,f>}{\<\vec{e}_f,\vec{e}_f>}
                        \,\frac{c_d^2}{\sqrt{pd}},
\]
where the $c_d$ are the Fourier coefficients of
$\Theta_{\id{p}}(\vec{e}_f) = \sum_{d \ge 1}c_d\,q^d$.
\end{mainthm}

This formula is the same as Conjecture $2$ in \cite{Pacetti-Tornaria},
with the difference of a factor of $8 \pi^2$ coming from a different
normalization in the Petersson inner product. The extra factor of
$\frac{p}{p-1}$ which shows up in \cite{Pacetti-Tornaria},
when $f$ is the quadratic twist of a level $p$ form, is due to the
fact that $L(f,D,s)$ we use here is not primitive at $p$ in this case.

The case of odd discriminant $D$ is proved using the results in
Section~\ref{sec:specialpoints}.
We avoid the technical difficulties of the case of
even discriminants by resorting to a theorem of Waldspurger, which
allows us to recover the formula for even discriminants from the case
of odd discriminants.

\smallskip
\noindent {\bf Acknowledgement:}
The first author would like to thank the Centro de Mate\-m\'a\-tica
of the Facultad de Ciencias for its hospitality during different visits.

\section{Quaternion algebras, bilateral ideals and special points}
\label{sec:bilateral}

Let $\H$ be a quaternion algebra over $\Q$, i.e. a central
simple algebra of dimension $4$ over $\Q$.  Given a place $v$ of $\Q$,
let $\H_v:= \H \otimes \Q_v$ be its completion at $v$. The algebra
$\H_v$ is either a division algebra or isomorphic to the algebra $M_2(\Q_v)$
of $2 \times 2$ matrices with coefficients in $\Q_v$. The place $v$ is
said to be \emph{ramified} if $\H_v$ is a division algebra and
\emph{split} if not. The number of ramified places is finite and
even. The discriminant of $\H$ is the product of all ramified primes
of $\Q$. 

We require $\H$ to be definite or ramified at $\infty$,
meaning $\H_\infty=\H\otimes\RR$ is a division algebra (the Hamilton
quaternions). The discriminant of $\H$ is thus the product of an
\emph{odd} number of primes.

It is well known
that $\H$ has an antiautomorphism called \emph{conjugation} and for
$\quat{x} \in \H$ we denote its action by $\bar{\quat{x}}$.
For $\quat{x}\in\H$ we define $\norm\quat{x} = \quat{x}
\bar{\quat{x}}$ and $\trace\quat{x} = \quat{x} + \bar{\quat{x}}$
the \emph{reduced norm} and \emph{reduced trace} of $\quat{x}$,
respectively.
The \emph{norm} of a lattice $\id{a}$ is defined as
$\norm\id{a}:=\gcd\set{\norm\quat{x}\st\quat{x}\in\id{a}}$.
We equip $\id{a}$ with the quadratic form
$\norm_{\id{a}}(\quat{x}):=\norm\quat{x}/\norm\id{a}$,
which is primitive; its determinant is a square,
and we denote its positive square root by $\disc(\id{a})$.
In particular, when $\O\subseteq\H$ is an order, $\disc(\O)$
is its \emph{$($reduced$)$ discriminant}.
The subscript $p$ will denote localization at $p$, namely
$\idp{a} := \id{a} \otimes \Z_p$.

Given a lattice $\id{a}$ in $\H$, its right order is given by
\[
   \Or(a) := \set{\quat{x}\in\H \st \id{a}\quat{x}\subseteq\id{a}}.
\]
The left order of $\id{a}$ is defined similarly.  If $\O$ is an order
in $\H$, we let $\Ix(\O)$ be the set of \emph{left $\O$-ideals},
i.e. the set of lattices $\id{a}\subseteq\H$ such that
$\idp{a}=\O_p\quatp{x}$ for every prime $p$, with
$\quatp{x}\in\Hpx$. Note that we do not define what an ideal in a
quaternion algebra is since we will only deal with left $\O$-ideals;
we recomend the reader to look at \cite{Vigneras} for such definition
and the theory of ideals in general. We just want to remark that the
condition of $\id{a}$ being locally principal is equivalent to
$\id{a}$ being a projective $\O$-module, which is the standard
definition of an ideal for a non-maximal order of a number field. It
is clear from the definition that if $\id{a}$ is a left $\O$-ideal,
its left order is just $\O$.

Let $\Mt(\O)$ be the vector space over $\RR$ with basis $\Ix(\O)$. Consider the 
\emph{height pairing},
\begin{equation}
  \<\id{a},\id{b}> :=\tfrac{1}{2}\,\numset{\quat{x}\in\Hx \st \id{a}\quat{x} = \id{b}}
  =
  \begin{cases}
     \frac{1}{2}\,\num{\Or(a)^\times} &  \text{if $\id{a}\quat{x} =
       \id{b}$, $x \in \H^\times$,}\\
     0                   &  \text{otherwise,}
  \end{cases}
\label{eq:heightpairing}
\end{equation}
as an inner product on it.  For a geometric interpretation of the
height pairing see \cite[\S4]{Gross}, where it is introduced as a
pairing on the Picard group of certain curves of genus zero.  In this
geometric context, special points (that will be defined in Section
\ref{sec:specialpoint}) should be regarded as analogues of Heegner
points on modular curves.

More generally, given $\O$-ideals $\id{a}$ and
$\id{b}$, define
\[
\Hom(\id{a},\id{b}) := \set{ u \in B^\times \st \id{a} u \subset \id{b}}.
\]
Then $\<\id{a},\id{b}> = \frac{1}{2} \# \{u \in \Hom(\id{a},\id{b}) \,
: \, \norm{u}=\frac{\normid{b}}{\normid{a}}\}$.

Let $\id{a}\in\Ix(\O)$, and $m\geq 1$ an integer. We set
\[
   \T_m(a) := \set{\id{b}\in\Ix(\O) \st
        \id{b}\subseteq\id{a}, \quad
        \norm\id{b}=m\,\norm\id{a} }.
\]

The \emph{Hecke operators} $\t_m:\Mt(\O)\rightarrow\Mt(\O)$ are
then defined by
\[
  \t_m\id{a}  := \sum_{\id{b}\in\T_m(a)} \id{b}
\]
for $m\geq 1$ and $\id{a}\in\Ix(\O)$, and extended by linearity to all
of $\Mt(\O)$.  
Moreover, we have
\[
  \<\id{a},\t_m\id{b}> =
  \frac{1}{2} \#\set{u\in \Hom(\id{a},\id{b}) \st \norm{u} =
  m\,\frac{\normid{b}}{\normid{a}}},
\]
and the Hecke operators are self-adjoint with respect to the height
pairing \cite[Proposition 1.3]{Pacetti-Tornaria}.

We define an equivalence relation on the set of left $\O$-ideals by
$\id{a},\id{b}\in\Ix(\O)$ are in the same class if
$\id{a}=\id{b}\quat{x}$, for some $\quat{x}\in\Hx$; we write $\cls{a}$
for the class of $\id{a}$. The set of all left $\O$-ideal classes,
which we denote by $\I(\O)$, is known to be finite. If we denote by
$\M(\O)$ the vector space over $\RR$ with basis $\I(\O)$, it has an inner
product and an action of Hecke operators by considering the quotient
map
\[
\Mt(\O) \twoheadrightarrow \M(\O).
\]
Note that $\I(\O)$ is an orthogonal basis of $\M(\O)$, and it is clear that
$\<\,,\,>$ is positive definite on $\M(\O)$.

The Hecke operators $t_m$ with $(m,\disc(\O))=1$ generate a
commutative algebra $\HeckeRing_0$, which is indeed generated by the
$t_p$ with $p\nmid\disc(\O)$ prime. Since the Hecke operators are
self-adjoint it follows, by the spectral theorem, that 
$\M(\O)$ has an orthogonal basis of common eigenvectors for
$\HeckeRing_0$.

We remark that $\M(\O)$ has a natural integral structure as the free
$\ZZ$-module spanned by $\I(\O)$ which is quite important; the height
pairing and the Hecke operators are defined over $\ZZ$.
However, the eigenvectors may not be defined over $\ZZ$, but only over
a totally real number field. For our purposes, using $\RR$ as the
field of coefficients will suffice.

The left $\O$-ideal $\id{a}$ is \emph{bilateral} if its right order is
also $\O$. The set of bilateral $\O$-ideals forms a group under ideal
multiplication. This group contains the principal ideals generated by
non-zero rational elements, which we denote by $\Q^\times$, in its
center, so we can consider bilateral $\O$-ideals modulo
$\Q^\times$-equivalence, namely two bilateral ideals $\id{a}$ and
$\id{b}$ are $\Q^\times$-equivalent if $\id{a}=\id{b}x$ with
$x\in\Q^\times$.  The group of bilateral $\O$-ideals modulo
$\Q^\times$-equivalence acts on $\M(\O)$ by left multiplication on the
basis elements. We denote $W_\id{m}$ the operator corresponding to the
bilateral $\O$-ideal $\id{m}$.

\begin{remark}
\label{remark:bilateral}
Note that this action commutes with the action
of $\HeckeRing_0$. Furthermore, these operators are unitary; hence,
those of order $2$ are self-adjoint.
On the other hand, the group of bilateral $\O$-ideals modulo $\Q^\times$ needs
not to be commutative (see Section~\ref{section:Ot}).
\end{remark}

We now give an ad\`elic reinterpretation of the ideal theory which will
allow us to work locally. Let $\hZZ=\prod_p{\ZZ_p}$ be the profinite
completion of $\ZZ$, and let $\hQQ:=\hZZ\otimes\QQ$ the ring of finite
ad\`eles of $\QQ$. Let also $\hH:=\H\otimes\hQQ$ and
$\hO:=\O\otimes\hZZ$. By the Eichler local-global principle for
lattices, we have a bijection between global lattices $\id{a}$ in $\H$ and
collections $\set{\idp{a}}$ of local lattices such that
$\idp{a}=\O_p$ for all $p$ except finitely many.
Hence, if $\id{a}$ is a left $\O$-ideal with $\idp{a}=\O_p\quatp{x}$, 
we have $(\quatp{x})\in\hHx$, and any
$(\quatp{x})\in\hHx$ determines an ideal in $\H$ in
that way. Furthermore, $\hOx$ acts on $\hHx$ by left multiplication,
and each ideal corresponds to a unique orbit for this action, i.e.
\[
\text{left $\O$-ideals}\quad\longleftrightarrow\quad\hOx\backslash \hHx.
\]

Note that if $\id{a}\leftrightarrow(\quatp{x})$, its right order
$\Or(a)$ is given locally by $\quatp{x}^{-1}\O_p\quatp{x}$. Hence, if
we set
\[
N(\hOx):=\set{(\quatp{x})\in\hHx \st \quatp{x}\O_p=\O_p\quatp{x}
\quad \forall p},
\]
we have the correspondence
\[
\text{bilateral $\O$-ideals}\quad\longleftrightarrow\quad\hOx\backslash N(\hOx).
\]
Moreover, if $p\nmid\disc(\O)$, we have that $\O_p$ is a maximal order 
in the split quaternion algebra over $\Q_p$ and its local
normalizer $N(\O_p^\times)=\O_p^\times \Q_p^\times$
(this is just a statement about the ring of 2 by 2 matrices with
coefficients in $\Z_p$, see~\cite[Proposition 1]{Eichler}). Hence the group of
bilateral $\O$-ideals modulo $\Q^\times$-equivalence,
$\hOx\backslash N(\hOx) / \Q^\times$, equals the finite local product
\[
\prod_{p\mid\disc(\O)} \O_p^\times\backslash N(\O_p^\times) / \Q_p^\times.
\]

\subsection{Special points}
\label{sec:specialpoint}
  Let $B$ be a quaternion algebra (over $\QQ$) of discriminant $N$,
  and let $K \subset B$ be a quadratic subfield of discriminant
  $D_0$. We note that
  \begin{itemize}
  \item If $B$ is definite, then $D_0<0$; and
  \item if $p\mid N$, then $\kro{D_0}{p}\neq 1$.
  \end{itemize}
  These are all the \emph{local} obstructions for such an embedding to
  exist; by Hasse's principle, there are no additional \emph{global}
  obstructions, i.e. if $K$ is an imaginary quadratic field of
  discriminant $D_0$ and for each prime number $p \mid N$, $\kro{D_0}{p}
  \neq 1$ then there exists an embedding of $K$ into $\H$.

The situation for orders is quite different. If $K$ embeds into $\H$,
let $\O$ be an order in $B$, and let $O:=\O\cap K$.
Then $O$ is an order in $K$ of discriminant $D=D_0 s^2$, for some
$s\in\ZZ$.
Conversely, given an order $O$ in $K$, 
there might be new local obstructions
for the existence of an embedding of $O$ into $\O$,
for example:
  \begin{itemize}
  \item If $\O$ is an order of discriminant $p^2$ (which are defined in \cite{Pizer2} and called ``orders of level $p^2$''), then $p\mid D$.
  \item If $\O$ is an order of level $p^2$ with character sign
    $\sigma$ (see Remark \ref{rem:char} for the definition), then
    $p\mid D$ and $\kro{D/p}{p}=\sigma$.
  \end{itemize}
Also, there are of course \emph{global} obstructions. However, if $O$
can be locally embedded in $\O$ at every place, it follows that $O$ can
be embedded in some order $\O'$ that is locally conjugate to $\O$
(i.e. in the same \emph{genus}). From now on we fix an order $\O$ in
the quaternion algebra $\H$. 

\begin{definition} If $O$ is an order in $K$ of discriminant $D$, a special
  point for $O$ is a pair $(\id{a},\psi)$, where $\id{a}$ is a left
  $\O$-ideal and $\psi:K \hookrightarrow \H$ is an embedding such that
  $\Or(\id{a}) \cap \psi(K) = \psi(O)$.
\end{definition}

This means that a special point of discriminant $D$ is an optimal
embedding of $O$ into the right order of an ideal $\id{a}$. Note that
this definition is slightly different from the one given in
\cite{Eichler}. There a special point is just a pair $(\psi,\O')$,
where $\O'$ is in the same genus of $\O$ and $\psi$ is an optimal
embedding of $O$ into $\O'$. Clearly given two orders $\O_1$, $\O_2$
in the same genus, there exists an ideal $\id{a}$ whose left order is
$\O_1$ and whose right order is $\O_2$, but there might be more than
one. The advantage of our definition will become clear while counting
classes of special points and their ad\`elic description.

The group $\H^\times$ has a right action on special points given for
$\alpha \in \H^\times$ by $(\id{a},\psi) \cdot \alpha =
(\id{a}\alpha,\alpha^{-1}\psi\alpha)$. If we fix a set of
representatives $\I(\O) =
\set{\id{a}_1,\ldots \id{a}_h}$ for the ideal classes,
any special point is equivalent to
$(\id{a}_i, \psi)$, where $\psi$ is an optimal embedding into
$\Or(\id{a}_i)$. Furthermore, $(\id{a}_i,\psi_1) \sim
(\id{a}_i,\psi_2)$ if and only if there exists $\alpha \in
\Or(\id{a}_i)^\times$ such that $\psi_1 = \alpha^{-1} \psi_2
\alpha$. Noting that if $(\id{a}_i,\psi)$ is a special point,
$\psi(O)^\times$ acts trivially, Eichler deduces the formula for the
number of non-equivalent special points of discriminant $D$ where $\O$
is an Eichler order (see \cite[Proposition 5]{Eichler} for Eichler orders
of square free level and \cite{Hijikata} for the general case). If $D$
is a fundamental discriminant, the formula reads
\begin{equation}
\label{eichler-formula}
\prod_{p \mid N}\left((1-\kro{D}{p}\right) \prod_{p \mid H} \left((1 + \kro{D}{p}\right) h(O_D),
\end{equation}
where $\H$ is the quaternion algebra with discriminant $N$ and $\O$ is
an Eichler order of discriminant $HN$.
The Kronecker symbols represent the condition for such
an embedding to exist.

We want to give an ad\`elic definition for the classes of special
points. Following the previous notation, let
$\widehat{O}:=O\otimes\widehat{\ZZ}$ and $\widehat{K}:=K\otimes\widehat{\QQ}$.
We will assume that $\widehat{K}\cap\widehat{\O}=\widehat{O}$.

Denote
by $\Ix(O)$ the set of $O$-(fractional) ideals and by $\I(O)$ the set
of $O$-ideal classes. We have the following ad\`elic interpretation:
  \par\vspace{2ex}

  \begin{tabular}{l@{\qquad}l@{\quad$\longleftrightarrow$\quad}l}
  $\O$-ideals:
    & $\Ix(\O)$
    & $\widehat{\O}^\times\backslash\widehat{B}^\times$
    \\
  $\O$-ideals modulo scalars:
    & $\Ix(\O)/\QQ^\times$
    & $\widehat{\O}^\times\backslash\widehat{B}^\times/\QQ^\times$
    \\
  $K$-points:
    & $\Ix(\O)/K^\times$
    & $\widehat{\O}^\times\backslash\widehat{B}^\times/K^\times$
    \\
  $\O$-classes:
    & $\I(\O)$ 
    & $\widehat{\O}^\times\backslash\widehat{B}^\times/B^\times$
    \\
  \\
  \end{tabular}

  \begin{tabular}{l@{\qquad}l@{\quad$\longleftrightarrow$\quad}l}
  $O$-ideals:
    & $\Ix(O)$
    & $\widehat{O}^\times\backslash\widehat{K}^\times$
    \\
  $O$-ideals modulo scalars:
    & $\Ix(O)/\QQ^\times$
    & $\widehat{O}^\times\backslash\widehat{K}^\times/\QQ^\times$
    \\
  $O$-classes:
    & $\I(O)$
    & $\widehat{O}^\times\backslash\widehat{K}^\times/K^\times$
    \\
  \end{tabular}

\medskip

Note that the height pairing defined at the beginning of this section
induces an inner product on the $\Z$-module spanned by $K$-points and
$R$-classes as well.

\begin{lemma}
  The embedding $K\subset B$ induces an injective map (by right multiplication)
  \[
     \Ix(O) \hookrightarrow \Ix(R).
  \]
\end{lemma}
  \begin{proof}
  The inclusion $K\subseteq B$, composed with a projection, induces a map
  $\widehat{K}^\times\rightarrow\widehat{B}^\times\rightarrow\widehat{R}^\times\backslash\widehat{B}^\times$,
  and it is clear that the kernel of this composite map is
  $\widehat{K}^\times\cap\widehat{R}^\times=\widehat{O}^\times$.
\end{proof}
This induces the following diagram:
  \[
  \xymatrix{
  \text{$O$-ideals} \ar@{^{(}->}[rr]\ar@{->>}[d] & & \text{$R$-ideals} \ar@{->>}[d]\\
  \text{$O$-classes} \ar@{^{(}->}[rr]\ar@{..>}[drr]  & & \text{$K$-points} \ar@{->>}[d]\\
  & & \text{$R$-classes}
  }
\label{diag1}
  \]
\medskip

\noindent where the horizontal arrows are injective, and the vertical arrows
surjective.  Note that, despite the $O$-classes and the $R$-classes
being independent of the embedding $K\subset B$, the dotted map does
indeed depend on the choice of embedding, as do the two horizontal maps.

\subsection{$O$-points} From now on, we fix an embedding $i:K
\hookrightarrow B$. For $x\in\Hh^\times$, we define $\Oh_x:=x^{-1}\Oh
x$, $R_x:=B\cap\Oh_x$, and $O_x:=K\cap\Oh_x$. Note that $R_x$ is an
order in $B$ (the \emph{right} order of $\Oh x$) locally conjugate to
$R$, and that $O_x$ is an order in $K$.  Note also that $R_x$ depends
only on the class of $x$ as an $R$-ideal modulo scalars, and that
$O_x$ depends only on the class of $x$ as a $K$-point.

\par 
If we set $\widehat{N}_O:=\set{x\in \H^\times \st O_x=O}$, then we define the
\emph{$O$-points} as the elements of $\Oh^\times\backslash\widehat{N}_O/K^\times$.

\begin{proposition} If $O$ is an order in $K$ of discriminant $D$, then the
classes of special points for $O$ are in one-to-one correspondence
with the $O$-points.
\label{prop:specialpointscorrespondence}
\end{proposition}
\begin{proof} Recall that we fixed an embedding $i:K \hookrightarrow
  \H$. If $(\id{a},\psi)$ is a special point of discriminant $D$, there
  exists $\alpha \in \H^\times$ such that $\alpha^{-1}\psi\alpha =
  i$. Furthermore, $\alpha$ is determined up to multiplication on the
  right by $K^\times$. Then the point $(\id{a},\psi) \sim
  (\id{a}\alpha,i)$. Since $\id{a}\alpha$ is a left $\O$-ideal, there
  exists $\quat{x} \in \Hh^\times$ such that $\Oh \quat{x} =
  \id{a}\alpha$. Then we associate to the pair $(\id{a},\psi)$ the
  $O$-point $\quat{x}$. It is immediate that $\quat{x}$ is in
  $\widehat{N}_O$ and that it is defined up to multiplication on the right by
  $K^\times$ and multiplication on the left by $\Oh^\times$, i.e. it is a point in
  the double quotient $\Oh^\times \backslash\widehat{N}_O/K^\times$.

Conversely, to an element $\quat{x} \in \widehat{N}_O$, we
associate the equivalence class of $(\Oh \quat{x},i)$. The condition
of the embedding being optimal is clear, and multiplication on the
right by $\Oh^\times$ give the same special point. We are left to
prove that right multiplication by $K^\times$ gives equivalent special
points, but this is clear since if $k \in K^\times$, conjugation by
$k$ acts trivially on $i$. 
\end{proof}

\subsection{Groups acting on $O$-points}
\begin{proposition}
    \label{quadraticaction}
\begin{enumerate}
  \item $\widehat{K}^\times$ acts on $\widehat{N}_O$ by right multiplication, i.e.
    $\widehat{N}_O\widehat{K}^\times\subseteq\widehat{N}_O$.
  \item \label{quadraticaction:2}
    $\widehat{O}^\times\backslash\widehat{K}^\times$ acts on
    $\Oh^\times\backslash\widehat{N}_O^\times$.
  \item The action in \eqref{quadraticaction:2} is free.
  \end{enumerate}
 \end{proposition}
 This action induces a free action of the group of $O$-classes on the set of $O$-points,
  and the space of orbits is $\widehat{R}^\times\backslash\widehat{N}_O/\widehat{K}^\times$.
  The \emph{canonical $O$-orbit} is the image of the map
  $\text{$O$-classes}\hookrightarrow\text{$K$-points}$.
  \begin{proof}[\proofname{} of Proposition~\ref{quadraticaction}]
\begin{enumerate}
\item indeed, if $a\in\widehat{K}^\times$, $\quat{x}\in\widehat{B}^\times$,
  then $O_{\quat{x}a}=K\cap
  a^{-1}\quat{x}^{-1}\widehat{R}\quat{x}a=a^{-1}(K\cap
  \quat{x}^{-1}\widehat{R}\quat{x})a=a^{-1}O_{\quat{x}} a=O_{\quat{x}}$.
\item if $a\in\widehat{O}^\times = K\cap \quat{x}^{-1}\Oh \quat{x}$ then
  $\quat{x}a\in\Oh\quat{x}$ and since $a$ is a unit, $\Oh \quat{x} =
  \Oh a\quat{x}$, which implies that $\widehat{O}^\times$ acts trivially
  on $\widehat{R}^\times\backslash\widehat{N}_O$.
\item if $\quat{x}a\in\widehat{R}\quat{x}$, then
$a\in\widehat{K}^\times\cap \quat{x}^{-1}\widehat{R}\quat{x} =
  \widehat{O}^\times$.
\end{enumerate}
\end{proof}

Recall that $N(\hOx) = \set{(\quat{x}_p) \st \quat{x}_p^{-1} \O_p
  \quat{x}_p = \O}$, modulo left multiplication by $\hOx$, corresponds
  to the group of bilateral $\O$-ideals. It clearly
acts on $\widehat{N}_O$ by left multiplication. Hence the group
$$G:=\hOx \backslash N(\hOx) / \QQ^\times \times \widehat{O}^\times \backslash
\widehat{K}^\times / K^\times$$ (i.e. the group of bilateral ideals modulo
$\Q^\times$-equivalence times the class group of $O$) acts on
the set of $O$-points.

\begin{proposition} If $\O$ is an Eichler order and $O$ is an order in $K$ of
  discriminant $D$, the action of $G$ on the classes of $O$-points is
  transitive. Furthermore, if $\gcd(\disc(\O),D)=1$ the action is
  free, while if $p$ is a prime that divides $\gcd(\disc(\O),D)$, the
  pair $(\id{p}_\O,\id{p}_O^{-1})$ acts trivially, where $\id{p}_{\O}$
  is the bilateral $\O$-ideal of norm $p$ and $\id{p}_O$ is the
  $O$-ideal of norm $p$ (in other words, the actions of $\id{p}_\O$
  and $\id{p}_O$ are the same.)
\label{transitive-action}
\end{proposition}
\begin{proof} That the action is transitive is a consequence of
  \cite[Proposition 3]{Eichler} in the case where $\O$ has square free
  discriminant, and by the work of \cite{Hijikata} for general
  levels. Once we know that the action is transitive, the second
  statement follows from \eqref{eichler-formula} (which gives the
  number of non-equivalent $O$-points) and from \cite[Proposition
    1]{Eichler} (which gives the number of bilateral ideals).
\end{proof}

Note that although in \cite{Eichler} only Eichler orders of square
free level are studied, most of the results proven there are also true
for all Eichler orders by the work of Hijikata
(\cite{Hijikata}). Furthermore, Proposition \ref{transitive-action}
holds also for orders of level $p^2M$ (that will be studied in the
next section), by \cite[Theorem 2.7 and Theorem 4.8]{Pizer2}.

\section{Orders of level $p^2$}\label{section:Ot}

Fix a prime $p>2$ and let $\H$ be the quaternion algebra over $\QQ$
which is ramified at $p$ and $\infty$.  Let $\Om$ be a maximal order in
$\H$ and let $\Ot$ be the unique order of index $p$ in $\Om$.
We note that
\[
\Ot = \set{\quat{x} \in \Om \st p \mid \normx\quat{x}},
\]
where $\normx\quat{x}:=(\trace\quat{x})^2-4\norm\quat{x}$ is the
discriminant of the characteristic polynomial of $\quat{x}$.

There is a quadratic character $\chi:\Ix(\Ot) \mapsto \{\pm1\}$, given
on $\id{a} \in \Ix(\Ot)$, by $\chi(\id{a}) =
\kro{\norm_{\id{a}}(\quat{x})}{p}$, where $x \in \id{a}$ is any element such
that $p \nmid \norm_{\id{a}}(\quat{x})$. Furthermore, the character
depends only on the equivalence class of $\id{a}$, i.e. it is a
character on $\I(\Ot)$ (see \cite[Proposition 5.1]{Pizer2}).

There is a $\HeckeRing_0$-equivariant bilinear map in $\M(\Ot)$ with
values in the space $M_2(p^2)$ of modular forms of weight $2$ and
level $p^2$ defined on the basis by
\[
\Hecke(\cls{a},\cls{b}) := \vartheta(\id{a}^{-1}\id{b}) =
\frac{1}{2} \sum_{x\in\id{a}^{-1}\id{b}}
q^{\norm\quat{x} / \normid{a^{-1}b}}
\]
This map induces a correspondence between eigenvectors in $\M(\Ot)$
and eigenforms of weight $2$ and level $p^2$. For an eigenform $f$ of
weight $2$ and level $p^2$ we denote by $\M(\Ot)^f$ the $f$-isotypical
component of $\M(\Ot)$, i.e. the eigenspace for the action of
$\HeckeRing_0$ with the same eigenvalues as $f$.
We have the following result due to Pizer (\cite[Theorem 8.2]{Pizer2}) : 
\[
  \dim \M(\Ot)^f = \begin{cases}
     1 & \text{if $f$ is an oldform,} \\
     1 & \text{if $f$ is the quadratic twist of a level $p$ form,}\\
     0 & \text{if $f$ is the non-quadratic twist of a level $p$ form,}\\
     2 & \text{if $f$ is not the twist of a level $p$ form.} \\
  \end{cases}
  \]
There is a natural inclusion $\M(\Om)\hookrightarrow\M(\Ot)$ which
is $\HeckeRing_0$-equivariant (see \cite[Proposition
1.14]{Pacetti-Tornaria}).

\begin{definition}
The space of old forms $\M(\Ot)^\old$ is the image of $\M(\Om)$ under 
the natural inclusion. Its orthogonal complement is denoted $\M(\Ot)^\new$ and
is called the new space.
\end{definition}

\begin{proposition}
The eigenvectors in $\M(\Ot)^\old$ correspond to eigenforms in
\linebreak
$M_2^\old(p^2)$,
and the eigenvectors in $\M(\Ot)^\new$
correspond to eigenforms in $M_2^\new(p^2)$.
\end{proposition}

\begin{proof} The first assertion is clear. For the second assertion,
    it is clear that if $f \in M_2^\new(p^2)$ and
   $\vec{e}_f \in \M(\Ot)^f$ then $\vec{e}_f$ is orthogonal to
   $\M(\Ot)^\old$. It might happen that there exists $\vec{e}
   \in \M(\Ot)^\new$ such that $f=\Hecke(\vec{e},\vec{e})$ is
   old.  But this is not the case, since if $f$ is an 
   old form, the $f$-isotypical component in $\M(\Ot)$ has dimension $1$. 
   Then, since $\M(\Om)^f$ is non-empty, we have $\M(\Om)^f = \M(\Ot)^f$.
\end{proof}

\subsection{Bilateral $\Ot$-ideals and its action on $\M(\Ot)$}

\begin{proposition} The group of bilateral $\Ot$-ideals modulo
  $\Q^\times$-equiva\-lence is isomorphic to the dihedral group $D_{p+1}$ of\/
  $2(p+1)$ elements. Furthermore the rotations correspond to
  bilateral $\Ot$-ideals of norm $1$ while the symmetries correspond
  to bilateral $\Ot$-ideals of norm $p$.
\label{prop:bilateral}
\end{proposition}

\begin{proof} See Proposition 9.26 of \cite{Pizer2}.
\end{proof}

By the remark in page~\pageref{remark:bilateral}, we are interested in
the operators corresponding to bilateral $\Ot$-ideals of order $2$.
By Proposition~\ref{prop:bilateral}, any norm $p$ bilateral
$\Ot$-ideal $\id{p}$ has order $2$.
There is also the one corresponding to the unique norm $1$ bilateral
$\Ot$-ideal $\id{m}_{\Ot}$ of order $2$,
which will be denoted by $\Wt$. This ideal commutes with any norm 
$p$ bilateral $\Ot$-ideal.

\begin{proposition}\label{prop:Wt}
    The operator $\Wt$ acts on $\M(\Ot)$ as the Atkin-Lehner
    involution $W_{p^2}$, i.e.
    \[
    \Hecke(\vec{u},\Wt\vec{v}) = \Hecke(\vec{u},\vec{v})|_{W_{p^2}}.
    \]
\end{proposition}
\begin{proof}
    In the basis of $\Ot$-ideal classes we have
    \[
    \Hecke(\cls{a},\Wt\cls{b}) = \vartheta(\id{a^{-1}}\,\id{m}_{\Ot}\,\id{b}) =
    \vartheta( (\id{a^{-1}}\,\id{m}_{\Ot}\,\id{a})\, \id{a^{-1}\,b} ).
    \]
    Clearly, $\id{a^{-1}}\,\id{m}_{\Ot}\,\id{a}$ is the unique norm 1
    bilateral $\Or(a)$-ideal of order $2$, and it thus follows
    from \cite[Theorem 9.20]{Pizer2} applied to the order $\Or(a)$
    that
    $\vartheta( (\id{a^{-1}}\id{m}_{\Ot}\id{a}) \id{a^{-1}\,b} )=
    \vartheta(\id{a^{-1}\,b})|_{W_{p^2}}$.
\end{proof}

\begin{corollary}
    Let $f$ be an eigenform of level $p$, and let $\e{f}$ be a
    non-zero eigenvector in $\M(\Ot)^f$. Then $\Hecke(\e{f},\e{f})$ is
    a non-zero multiple of
    \[
        f(z) - \epsilon_p\, p f(pz),
    \]
    where $\epsilon_p$ is the sign of the Atkin-Lehner involution at
    $p$, i.e. $f|_{W_p} = \epsilon_p f$.
\label{oldforms}
\end{corollary}
\begin{proof} Since $f$ is an old form, $\vec{e}_f \in
    \M(\Ot)^\old$.
    Locally, $\Wt$ amounts to left multiplication by the element of
    order $2$ in $\Ot_p^\times \backslash \Om_p^\times$, which is
    clearly trivial in $\M(\Ot)^\old$.
    Thus $\Wt\e{f}=\e{f}$, and by the
    Proposition
    \[
    \Hecke(\e{f},\e{f})|_{W_{p^2}} = \Hecke(\e{f},\e{f}).
    \]
    The statement follows from the fact that $W_{p^2}$ acts on the
    oldspace generated by $f(z)$ and $pf(pz)$ by the matrix
    $\smat{0&\epsilon_p\\\epsilon_p&0}$.
\end{proof}

\subsection{Multiplicity one}

Let $\O, \O'$ be orders and $\id{a}$ be a left $\O$-ideal. We define

\[
\Psi^{\O}_{\O'}(\id{a}) := \set{\id{b}\in\Ix(\O') \st \id{b}\subseteq\id{a}, \quad
  \norm\id{b}=\norm\id{a} }.
\]

If $\id{a} \in \Ix(\Om)$, then the subgroup of rotations acts
transitively (although not necessarily faithfully) on the set
$\Psi^{\Om}_{\Ot}(\id{a})$ (see Section $3.3$ of \cite{Pacetti-Villegas}).

Let $\rho$ denote a generator of the norm
$1$ bilateral $\Ot$-ideals, then we have $\chi(\rho) = -1$ since the bilateral
$\Ot$-ideals of norm $1$ are exactly $\Psi^{\Om}_{\Ot}(\Om)$ and half of these
ideals have positive character while the other half have negative
character. Also $W_{\rho^2}$ preserves characters.

\begin{lemma}\label{lemma:Wrho}
    If $\e{}\in\M(\Ot)$ is an eigenvector for $W_\rho$ with eigenvalue
    $\epsilon=\pm1$, then $\phi(\e{},\e{})$ is
    \begin{enumerate}
        \item an oldform, if $\epsilon=1$;
        \item the quadratic twist of a level $p$ form, if
            $\epsilon=-1$.
    \end{enumerate}
\end{lemma}
\begin{proof} Let $\e{f}$ be an element in $\M(\Ot)^f$. Then $\e{f}$
can be written as
\begin{multline*}
\sum_{\cls{a} \in \I(\Om)} \left(\sum_{\set{\cls{b} \in
    \Psi^{\Om}_{\Ot}(\id{a}) \st \chi(\cls{b}) = 1}/ \sim}
    n_{\cls{b}} \cls{b}\right.
    \\
    \left.
    + \sum_{\set{\cls{b} \in
    \Psi^{\Om}_{\Ot}(\id{a}) \st \chi(\cls{b}) = -1}/ \sim}
    n_{\cls{b}} \cls{b}\right)
\end{multline*}
Since $(W_{\rho})^2=W_{\rho^2} = 1$ and $\rho^ 2$ acts transitively on
both sets of the previous sum, we get that
\begin{align*}
\e{f} = & \sum_{\cls{a} \in \I(\Om)} n_{\cls{a}} \sum\nolimits_{\set{\cls{b} \in
    \Psi^{\Om}_{\Ot}(\id{a}) \st \chi(\cls{b}) = 1}/ \sim}
     \cls{b} + \\
     & \sum_{\cls{a} \in \I(\Om)} m_{\cls{a}} \sum\nolimits_{\set{\cls{b} \in
    \Psi^{\Om}_{\Ot}(\id{a}) \st \chi(\cls{b}) = -1}/ \sim}
     \cls{b}.
\end{align*}

Since $\chi(\rho) = -1$, $W_{\rho}$ permutes the set on the first sum
with the set on the second one; since $\e{f}$ is an eigenvector
of $W_{\rho}$ with eigenvalue $\epsilon =\pm1$, then $n_{\cls{a}} =
\epsilon m_{\cls{a}}$ for all ideals $\cls{a} \in \I(a)$. Clearly if
$\epsilon =1$, then $\e{f}$ is in $\M(\Ot)^\old$ and
$\e{f}$ is as in Corollary \ref{oldforms} i.e. it has
the same eigenvalues as a weight $2$ and level $p$ form.
On the other hand, if $\epsilon =-1$, then $\e{f}$ is the twist of an
eigenvector in $\M(\Ot)^\old$ (see \cite{Pa-To2})
therefore it corresponds to a quadratic twist of a level $p$ form, as
claimed.
\end{proof}

\begin{theorem}\label{thm:multone}
    Let $\id{p}$ be a norm $p$ bilateral $\Ot$-ideal. The algebra
    $\HeckeRing_{0,\id{p}}$ generated by
    $\HeckeRing_0$ and $W_{\id{p}}$ is
    commutative, its action on the space $\M(\Ot)$ is semisimple, and
    its eigenspaces have multiplicity one.
\end{theorem}
\begin{proof}
    By the remark in page~\pageref{remark:bilateral},
    we have that
    $\HeckeRing_{0,\id{p}}$ is commutative and its elements are
    self-adjoint with respect to the height pairing.

    We will now prove multiplicity one.
    Let $f$ be a modular form of weight $2$ and level $p^2$, and
    consider the eigenspace $\M(\Ot)^f$, which we know has dimension
    at most $2$. Since $W_{\id{p}}^2=1$, we will assume 
    $W_{\id{p}} = \pm1$ on $\M(\Ot)^f$, as otherwise the statement is
    clearly true.

  In this case, we have the identity (on $\M(\Ot)^f$)
  \[
    W_{\rho} = W_{\rho \id{p}} W_{\id{p}}
             = W_{\id{p}} W_{\rho \id{p}}
             = W_{\rho^{-1}}
  \]
(the middle equation comes from the fact that $W_{\id{p}} = \pm1$.)
Hence $W_{\rho}^2 = 1$ on $\M(\Ot)^f$, and there is an eigenvector
$\e{}\in\M(\Ot)^f$ for $W_\rho$ with  eigenvalue $\pm 1$.
By the Lemma, $\phi(\e{},\e{})$ is either an oldform or the quadratic
twist of a level $p$ form, whose eigenspace for $\HeckeRing_0$ is
already one-dimensional.
\end{proof}

\begin{corollary}
    The action of $W_\id{p}$ on $\M(\Ot)$ gives an orthogonal
    decomposition
    \[
    \M(\Ot) = \ker(W_\id{p}-1) \oplus \ker(W_\id{p}+1),
    \]
    and the action of $\HeckeRing_0$ on each of the components
    has multiplicity one.
\end{corollary}
\begin{proof}
    The claimed decomposition is clear because $W_\id{p}^2 = 1$ on
    $\M(\Ot)$, and the multiplicity one on each of the components
    follows from the theorem.
\end{proof}

\begin{remark}\
    \begin{enumerate}
        \item
    The action of $W_\id{p}$ on $\M(\Ot)^\old$ is given by
    left multiplication by $\id{p}$ acting on $\M(\Om)$.
    When $f$ is an eigenform of level $p$, this action on
    $\M(\Om)^f$ is known to be $-\epsilon_p(f)$,
    where $\epsilon_p(f)$ is the eigenvalue of the Atkin-Lehner
    involution $|_{W_p}$. 
    Hence, $\M(\Ot)^f\subseteq \ker(W_\id{p} +\epsilon_p(f))$.
\item
    When $f$ is the quadratic twist of an eigenform $g$ of level $p$,
    one can see that the operator $W_\id{p}=-\epsilon_p(g)
    \chi(\id{p})$ where $\chi(\id{p})$ is the character of the left
    $\Ot$-ideal $\id{p}$.  Basically, if $\Phi$ denotes the twisting
    operator in $\M(\Om)$, given on a basis element $\cls{b} \in
    \I(\Ot)$ by $\Phi(\cls{b}) = \chi(\id{b}) \cls{b}$, it amounts to
    see that $W_\id{p}\Phi=\chi(\id{p})\Phi W_\id{p}$.  \par Hence,
    $\M(\Ot)^f\subseteq \ker(W_\id{p} +\epsilon_p(g)\chi(\id{p}))$.
    \end{enumerate}
\end{remark}

Recall from~\cite{Pacetti-Tornaria} there is a $\TT_0$-linear map
\[
\Theta:\M(\Ot)\rightarrow M_{3/2}(4p,\charp),
\]
where $\charp(n):=\kro{p}{n}$ is the quadratic character modulo $p$ or
$4p$ (according to whether $p\equiv 1$ or $3 \pmod{4}$, respectively).
This map can be defined as
\[
\Theta(\cls{\Ot}) := \frac{1}{2} \sum_{x\in\Ot/\ZZ} q^{-\Delta{x}/p},
\]
and extended to all of $\M(\Ot)$ by conjugation, i.e.:
\[
\Theta(\cls{b}) := \Theta(\cls{\id{b}^{-1}\Ot\id{b}}).
\]
Note that $\Theta(\cls{\rho\id{b}})=\Theta(\cls{\id{b}})$, since
$\rho^{-1}\Ot\rho=\Ot$. In other words the diagram
\[
\xymatrix@R=0.3ex@C=3em{
\M(\Ot) \ar[dr]^{\Theta}
        \ar[dd]_{W_\rho} \\
         & M_{3/2}(4p,\charp) \\
        \M(\Ot) \ar[ru]_{\Theta}
}
\]
is commutative.
\begin{proposition}\label{prop:theta}
    If $\vec{e}\in\M(\Ot)$ is an eigenvector for $\TT_0$, and
    $\Theta(\vec{e})\neq 0$, then $\vec{e}$ is old.
\end{proposition}
\begin{proof}
    If $\Theta(\vec{e})\neq 0$, we may assume that
    $\vec{e}\in\ker\Theta^\perp$, since $\Theta$ is $\TT_0$-linear.

    Since $\Theta W_\rho = \Theta$, it follows that
    $\ker\Theta$ is invariant by $W_\rho$.
    Moreover, $\ker\Theta^\perp$ is also invariant by $W_\rho$,
    because $W_\rho$ is unitary.

    Hence, $W_\rho(\vec{e})\in\ker\Theta^\perp$, and the vector
    \[
        \vec{e} - W_\rho(\vec{e})
    \]
    is both in $\ker\Theta^\perp$ and in $\ker\Theta$.
    Thus $\vec{e}$ is an eigenvector for $W_\rho$ with eigenvalue $1$,
    and the result follows from Lemma~\ref{lemma:Wrho}.
\end{proof}

\subsection{Suborders and symmetries}

In view of Proposition~\ref{prop:theta}, the map $\Theta$ is trivial
on $\M(\Ot)^\new$. In order to obtain non-zero half-integral weight
modular forms corresponding to new vectors in $\M(\Ot)$, we need to
work with the orders of index $p$ in $\Ot$, which we call \emph{orders
  of level $p^2$} (see \cite{Pacetti-Tornaria}). It is known that they
play a very important role in the theory of Shimura correspondence for
level $p^2$.

\begin{remark}
As suggested by the referee, we want to clarify the terminology
here. In \cite{Pizer2}, the author uses the name ``orders of level
$p^2$'' for $\Ot$. In his context, this was the natural term to use,
since the main achievement was to construct bases of integral weight modular
forms and, using the trace formula, he proves that weight $2$ modular
forms of level $p^2$ appear in $\M(\Ot)$. However, since orders in
quaternion algebras are in correspondence with ternary quadratic
forms, it is more natural to index orders by the level of the
corresponding ternary form, which was the definition we used in
\cite{Pacetti-Tornaria} and we maintain in this article. We hope this
will make no confusion to the reader.
\end{remark}

We recall the following
\begin{proposition}\label{prop:suborderlevelp}
  Let $L \subset \Ot$ be a lattice such that $[\Ot:L] = p$. Then
  $L$ is an order if and only if $\Z + p\Om \subset L$.
\end{proposition}
\begin{proof}
    This is Proposition 2.2 of \cite{Pacetti-Tornaria}.
\end{proof}

\begin{proposition} If $\id{p}$ is a norm $p$ bilateral $\Ot$-ideal
  then $\ZZ+\id{p}$ is an order of index $p$ in $\Ot$.
  Conversely, if $\Op$ is an order with $[\Ot:\Op]=p$
  then $L=\set{x\in\Op\st p\mid\trace\quat{x}}$ is a
  norm $p$ bilateral $\Ot$-ideal.
\end{proposition}
\begin{proof} Since $\id{p}$ has norm $p$ in $\Ot$, by a local
    computation it follows that $\id{p}\subseteq\Ot$, and
    $[\Ot:\id{p}]=p^2$.
    Also it is easy to check that $p \Om \subset \id{p}$,
    hence $\ZZ +p\Om \subset \ZZ+\id{p}$ and $[\Ot:\ZZ+\id{p}]=p$.
    By Proposition~\ref{prop:suborderlevelp} it follows that $\ZZ+\id{p}$
    is an order.
    \par
    For the converse, note that for $\quat{x}\in\Ot$ we have
    $p\mid\trace\quat{x}$ if and only if $p\mid\norm\quat{x}$.
    Hence $L=\set{x\in\Op\st p\mid\norm\quat{x}}$, and it follows
    that $\Op L \subseteq L$. Therefore, the left order of $L$ is
    either $\Om$, $\Ot$ or $\Op$. From the fact that all the lattices
    with one of them as left order are locally principal (see \cite{Bre})
    and $[\Ot:L]=p^2$, it follows that $L$ is a left $\Ot$-ideal of
    norm $p$, hence bilateral.
\end{proof}

\begin{remark}
\label{rem:char}
Recall the definition of the character $\chi$ in an order $\Op$ of index $p$
in $\Ot$:
\[
   \chi(\Op) := \kro{ - \Delta x / p}{p},
\]
where $x\in\Op$ such that $p\parallel \Delta x$. This is well defined
by \cite[Lemma 2.3]{Pacetti-Tornaria}, and it's clear that
\[
   \chi(\ZZ+\id{p}) = \chi(\id{p}),
\]
i.e. the correspondence preserves the character.
\end{remark}

\subsection{Theta series of weight $3/2$} \label{thetacorr}
Using the correspondence given in the previous section we define,
for $\id{p}$ a bilateral $\Ot$-ideal of
norm $p$, a map
\[
\Theta_\id{p} : \M(\Ot) \rightarrow M_{3/2}(4p^2,\charp)
\]
by
\[
\Theta_\id{p}(\cls{\Ot}) :=
\Theta(\cls{\ZZ+\id{p}})
=
\frac{1}{2} \sum_{\substack{x\in\ZZ+\id{p}/\ZZ}}
    q^{-\Delta x/p}.
\]

This definition extends to all of $\M(\Ot)$ by conjugation,
namely:
\[
\Theta_\id{p}(\cls{b}) := \Theta_{\id{b}^{-1} \id{p}
\id{b}}(\cls{\id{b}^{-1}\Ot\id{b}}).
\]

In \cite{Pacetti-Tornaria}, for $\id{b} \in \I(\Ot)$ and $\Op=\ZZ +
\id{p}$, it is defined
\[
\Theta_{\Op}(\cls{b}) := \Theta(\Or(\id{c})),
\]
where $\id{c}$
is any left $\Op$-ideal with index $p$ in $\id{b}$.

\begin{lemma} If $\Op = \ZZ + \id{p}$ then $\Theta_{\id{p}} = \Theta_{\Op}$.
\end{lemma}
\begin{proof} If $\id{b} \in \I(\Ot)$ and $\id{c}$ is any left
  $\Op$-ideal with index $p$ in $\id{b}$, we claim that $\Or(\id{c}) =
  \ZZ +\id{b}^{-1} \id{p} \id{b}$, which implies the assertion.

    To see it, we prove equality of both orders at the different
    completions. For primes $q \neq p$, the statement is trivial,
    since $\id{c}_q = \id{b}_q$ and
    $\Or(\id{c}_q) = \Or(\id{b}_q) = \id{b}_q^{-1} \id{b}_q = (\ZZ+
    \id{b}^{-1} \id{p} \id{b})_q$.
    
    At $p$, if $\idp{b} = \Ot_p x_p$, we can take $\idp{c} :=
    (\ZZ_p+\idp{p}) x_p$ (the global lattice $\id{c}$ with this local
    completions satisfies the hypothesis). Then $\Or(\idp{c}) = \ZZ_p + x_p^{-1}
    \idp{p} x_p = \ZZ_p + \idp{b}^{-1} \idp{p} \idp{b}$ as claimed.
\end{proof}

Note that
\[
\Theta_\id{p} (\cls{\rho\id{b}}) = \Theta_{\rho^{-1} \id{p} \rho}(\cls{b}) =
\Theta_{\id{p}\rho^2}(\cls{b}),
\]
hence for any rotation $\rho^k$ we have:
\[
\xymatrix@R=0.3ex@C=3em{
\M(\Ot) \ar[dr]^{\Theta_{\id{p}\rho^{2k}}}
        \ar[dd]_{W_{\rho^k}} \\
         & M_{3/2}(4p^2) \\
        \M(\Ot) \ar[ru]_{\Theta_\id{p}}
}
\]

\begin{proposition} The image of the map $\Theta_\id{p}(\M(\Om)^f)$ depends
    only on the character of $\id{p}$ for any eigenform $f$.
\end{proposition}

\begin{proof}
If $\id{p}$ and $\id{p}'$ are two bilateral ideals of norm $p$ with
the same character, then they differ by
the square of a rotation, say $\id{p} = \id{p}' \rho^{2k}$.
Since the operator $W_{\rho^k}$
commutes with the Hecke operators, the space $\M(\Ot)^f$ is
invariant under $W_{\rho^k}$, hence the statement follows from the previous
observation.
\end{proof}

Using the same argument for $\Wt$, since $\Wt^2=1$, we have the
commuting triangle
\[
\xymatrix@R=0.3ex@C=3em{
\M(\Ot) \ar[dr]^{\Theta_{\id{p}}}
        \ar[dd]_{\Wt} \\
         & M_{3/2}(4p^2) \\
        \M(\Ot) \ar[ru]_{\Theta_\id{p}}
}
\]
hence $\ker(\Wt+1)\subseteq \ker\Theta_\id{p}$; in other words,
eigenvectors corresponding to 
modular forms $f$ with $\epsilon(f)=-1$ have trivial image under
$\Theta_\id{p}$.

A similar computation shows that
\[
\Theta_\id{p} (\cls{pb}) = \Theta_{\id{p}^{-1} \id{p} \id{p}}(\cls{b}) =
\Theta_{\id{p}}(\cls{b}),
\]
thus
\[
\xymatrix@R=0.3ex@C=3em{
\M(\Ot) \ar[dr]^{\Theta_{\id{p}}}
\ar[dd]_{W_\id{p}} \\
         & M_{3/2}(4p^2) \\
        \M(\Ot) \ar[ru]_{\Theta_\id{p}}
}
\]
and again we have $\ker(W_\id{p}+1)\subseteq \ker\Theta_\id{p}$.

Summarizing, we have
\begin{proposition}
    In the irreducible components of $\M(\Ot)$ where $\Theta_\id{p}$ is
    non-zero, we have $W_\id{p}=\Wt = 1$.
    In particular, the image $\Theta_\id{p}(\M(\Om)^f)$, for an
    eigenform $f$, is at most $1$-dimensional.
\label{prop:Wp-action}
\end{proposition}

\section{Special points for level $p^2$}
\label{sec:specialpoints}

Following Section~\ref{sec:bilateral}, we fix $D<0$ an odd fundamental
discriminant.  We require $p\mid D$, since there are no special points
of discriminant $D$ for $\Ot$ otherwise.  Let $K$ be the imaginary
quadratic field of discriminant $D$, and $\Om_K$ its ring of
integers. Let $u_D$ be half the number of units in $\Om_K$, i.e. $u_D
= \frac{1}{2}\# \Om_K^\times$.  Write $D=p^\ast D_0$, where $p^\ast =
\kro{-1}{p}p$.

Fix a rational prime $q>0$ satisfying the conditions:
\begin{itemize}
\item $q \nmid 2D$.
\item $\kro{-q}{p}=-1$.
\item $q \equiv -1 \pmod{D_0}$.
\end{itemize}
By quadratic reciprocity, these conditions imply that $q$ is split in
$K$. We fix an ideal $\id{q}$ of $\Om_K$ of norm $q$, and
note that its genus $\gen{q}$ is the only element of the set
\[
Q = \set{ \gen{q} \st \normid{q}\equiv -p^2\pmod{D_0}}
\]
appearing in Theorem~\ref{thm:rankin} below.

Let $B = K + Kj$ with
\[
j^2 = -q ,
\]
and $j k = \bar{k} j$ for all $k\in K$, where $\bar{k}$ is the
complex conjugate of $k$.
\par
\begin{proposition}
$B$ is a quaternion algebra ramified precisely at $p$ and $\infty$.
\end{proposition}
\begin{proof} Clearly, $B$ is a quaternion algebra over $\QQ$, so
  we just need to find the set of
  ramified primes. In the basis $\set{1,\sqrt{D},j,\sqrt{D}\,j}$,
  the norm form is
  \[
     \Norm(x_0+x_1\sqrt{D}+x_2 j +x_3\sqrt{D}\,j)
     =
     x_0^2 - {D} x_1^2 + q x_2^2 - {D} q x_3^2.
  \]
  Since the norm form in $B$ is positive definite,
  $B$ ramifies at infinity.
  To check whether $B$ is ramified at a prime $l$ or not,
  we need to see if the norm form represents $0$ in $\Q_l$ for each
  prime $l$. Consider the different cases:
\begin{itemize}
\item If $l \nmid 2Dq$ then it is clear that $B$ is split at $l$,
    since the discriminant of the norm form is an $l$-adic unit in
    this case.
\item If $l = p$, the norm form represents zero if and only if
  $\kro{-q}{p} = 1$. The second condition on $q$ assures that this
  is not the case, hence $B$ ramifies at $p$.
\item If $l \mid D$ but $l \neq p$ then the norm form represents zero
  if and only if $\kro{-q}{l} = 1$, which is clearly the case since $-q
  \equiv 1 \pmod l$.
\item If $l = q$, the norm form represents zero if and only if
  $\kro{D}{q}=1$ which is the case by quadratic reciprocity. In fact,
  the conditions in $q$ imply $\varepsilon_{D_0}(-q)=+1$ and
  $\varepsilon_{p^\ast}(-q)=-1$, hence $\varepsilon_D(-q)=-1$.
  Since $D<0$ it follows that $\kro{D}{q}=1$ as claimed.
\end{itemize}
Since the number of ramified primes for a quaternion algebra is even,
we do not need to consider the prime $2$.
\end{proof}

Let $\id{D}_0$ be the ideal in $\Om_K$ of norm $D_0$, and $\id{p}_K$ the
ideal of $\Om_K$ of norm $p$. To simplify the notation, in this
section only we will omit the subscript $K$ writing $\id{p}=\id{p}_K$.
Define
\[
\Om := \set{\alpha + \beta j \st
  \alpha \in \id{D}_0^{-1},\;
  \beta \in \id{D}_0^{-1} \id{q}^{-1},\;
  \alpha-q \beta \in \Om_{K}},
\]
and
\[
\Ot := \set{\alpha + \beta j \st
  \alpha \in \id{D}_0^{-1},\;
  \beta \in \id{D}_0^{-1} \id{pq}^{-1},\;
  \alpha-q \beta \in \Om_{K}},
\]
This is consistent with the notation of the previous section by the
following theorem.

\begin{theorem} $\Om$ is a maximal order in $B$ and
  $\Ot$ is the unique order of index $p$ in $\Om$. 
\end{theorem}
\begin{proof} To prove that $\Om$ is an order, since $1 \in \Om$ and
    $\Om$ is closed under addition, we just need to check it is closed
    under multiplication.
    Let $a_1+b_1j,\, a_2+b_2j \in \Om$, then 
\[
(a_1+b_1j)(a_2+b_2j)=  (a_1a_2-qb_1\bar{b_2})+(a_1b_2+\bar{a_2}b_1)j.
\]
To prove that this is in $\Om$, we claim that it belongs to
  $\Om_l:= \Om \otimes \Z_l$ for all primes~$l$. Consider the cases:
\begin{itemize}
\item If $l \nmid D_0$ the claim is clear, since in this case
    \[
    \Om_l = (\Om_{K} + \id{q}^{-1} j) \otimes \ZZ_l,
    \]
    with $j^2 = -q$.
\item If $l\mid D_0$, then 
\[
a_1a_2-qb_1\bar{b_2}=a_1(a_2-qb_2) + qb_2(a_1-qb_1)+qb_1(qb_2-\bar{b_2}).
\]
The first two terms clearly belong to $\id{D}_0^{-1}\otimes\ZZ_l$. The last also
belongs to $\id{D}_0^{-1}\otimes\ZZ_l$ since $q \equiv -1 \pmod l$, and $b+\bar{b}
\in \Om_K\otimes\ZZ_l$ for all $b \in \id{D}_0^{-1}\otimes\ZZ_l$.
\par
Analogously, 
\[
(a_1b_2+\bar{a_2}b_1) = (a_1-b_1q)b_2+b_1(b_2q-a_2)+b_1(a_2+\bar{a_2}),
\]
and the same reasoning applies.
\par
Finally, the proof that
\[
(a_1a_2-qb_1\bar{b_2})-q(a_1b_2+\bar{a_2}b_1)j \in \Om_K\otimes\ZZ_l,
\]
follows from a similar computation.
\end{itemize}

The proof that $\Ot$ is an order is the same, except for $l=p$, where
$\Ot_p = (\Om_K+\id{p} j)\otimes\ZZ_p$ and the claim is clear.
Also this shows that $\Ot$ has index $p$ in $\Om$.

It remains to prove that $\Om$ is maximal,
or equivalently, that its reduced discriminant is $p$.
We compute the $l$-valuation of the
discriminant for each prime $l$:
\begin{itemize}
\item If $l \nmid Dq$ then $\Om_l=(\Om_k + \Om_Kj)\otimes \Z_l$, and
  the discriminant of the norm form in $\Om_l$ is an $l$-adic unit.
\item If $l \mid D_0$, $\Om_l \subset (\id{D}_0^{-1} +
    \id{D}_0^{-1} j)\otimes \ZZ_l$ with index $l$. The discriminant of
    the norm form in the latter is $l^{-2}$ since $l \nmid \norm{j}=q$
    hence the discriminant of the norm form in $\Om_l$ is an $l$-adic
    unit.
\item If $l=q$, $\Om_q = (\Om_K + \id{q}^{-1} j) \otimes \ZZ_q$. Since
    $\norm{j}=q$, the discriminant of the norm form in $\Om_q$ is a
    $q$-adic unit.
\item If $l=p$, $\Om_p = (\Om_K+\Om_Kj)\otimes \Z_p$ hence the
discriminant of the norm form in $\Om_p$ is $p^2$ since $p \mid D$.
\end{itemize}
\end{proof}

\subsection{Counting special points}

Recall that
 \[
\<\Ot\id{b},t_{m} \Ot\id{ab}> = \frac{1}{2} \#
\Hom(\Ot\id{b},\Ot\id{a}\id{b})[m],
\]
where
\[
\Hom(\Ot\id{b},\Ot\id{a}\id{b})[m]
:=
\set{u\in
\Hom(\Ot\id{b},\Ot\id{a}\id{b})
\st
\norm u = m \normid{a}}.
\]

Let $\calD$ be the set of ideals, 
\[
\calD:=\set{ \id{d} \st \normid{d} \mid D_0}.
\]

Note that the elements of $\calD$ are in one to one correspondence
with the elements of order $1$ or $2$ of the class group $\I(\Om_K)$,
since $D$ is odd and hence squarefree.

\begin{lemma} \label{lemma:homab}
Let $\id{a},\id{b}$ ideals of $\Om_K$ of norm prime to $D$,
and let $\id{d} \in \calD$. Then
\[
\begin{split}
\Hom(\tilde \Om \id{bd}, \tilde \Om \id{abd}) = 
\bigl\{\alpha + \beta j \; : \; & \alpha \in
\id{D}_0^{-1} \id{a},\,
\beta \in \id{D}_0^{-1} \id{pq}^{-1}\id{b}^{-1} \bar{\id{a}}\bar{\id{b}},\\
  & \alpha+q \beta \in (\Om_{K})_l \quad \forall \, l \mid \normid{d},\\
  & \alpha-q \beta \in (\Om_{K})_l \quad \forall \, l \mid D_0 \text{
  and } l \nmid \normid{d}
\bigr\}.
\end{split}
\]
\end{lemma}
\begin{proof} By definition, $\Hom(\Ot\id{bd},\Ot\id{abd}) =
    (\id{bd})^{-1} \Ot
\id{abd}$, i.e.
\begin{multline*}
\Hom(\Ot\id{bd},\Ot\id{abd}) =
\\
\set{b_0(\alpha + \beta j)ab_1 \st
  (\alpha + \beta j) \in \tilde \Om, b_0 \in
  (\id{bd})^{-1}, b_1 \in \id{bd} \text{ and } a \in \id{a}}.
\end{multline*}
\par
For $\alpha \in K$, $\alpha j = j \bar{\alpha}$ thus 
$b_0(\alpha + \beta j)ab_1= ab_0b_1 \alpha + \bar{a}b_0\bar{b_1}\beta
j$. The first term lies in $\id{D}_0^{-1}\id{a}$ while the second one
lies in $\id{D}_0^{-1} \id{q}^{-1} \id{b}^{-1} \bar{\id{a}}
\bar{\id{b}}$ since $\id{d}^{-1}\bar{\id{d}}=\Om_K$.
\par
We claim that $ab_0b_1\alpha
-q\bar{a}b_0\bar{b_1}\beta \in \Om_K \otimes \ZZ_l$ for all primes $l
\mid D_0/d$. Indeed
\[
ab_0b_1\alpha -q\bar{a}b_0\bar{b_1}\beta =
ab_0b_1(\alpha-q\beta)+q\beta b_0(ab_1-\overline{ab_1}),
\]
so the claim follows from the condition on the norms of $\id{a}$
and $\id{b}$, and the definition of $\Ot$.
\par
On the other hand, if $l\mid d$, then
\[
ab_0b_1\alpha +q\bar{a}b_0\bar{b_1}\beta =
ab_0b_1(\alpha-q\beta)+q\beta b_0(ab_1+\overline{ab_1}).
\]
The first term is in $\Om_K$ as before,
and since $b_1\in\id{bd}$, $l\mid ab_1+\overline{ab_1}$ so
the second term lies in $\Om_K\otimes\ZZ_l$ finishing the proof.
\end{proof}

We denote $\delta(n):=2^t$, where $t$ is the number
of prime factors of $\gcd(n,D_0)$. This is relevant because of the
following computation.

\begin{lemma} \label{lemma:sumd}
    With the same notation as above,
    let $\alpha \in \id{D}_0^{-1} \id{a}$
and
$\beta \in
\linebreak
\id{D}_0^{-1} \id{p}\id{q}^{-1}\id{b}^{-1} \bar{\id{a}}\bar{\id{b}}$
such that $\norm(\alpha+\beta j)\in\ZZ$,
and set $n=\frac{q\norm\beta\,\abs{D_0}}{p\norm{a}}\in\ZZ$.
Then
\[
\#\set{\id{d} \st \id{d} \in \calD ,\;
\alpha+\beta j\in
\Hom(\tilde \Om \id{bd}, \tilde \Om \id{abd}) } = \delta(n).
\]
\end{lemma}
\begin{proof}
Take a prime $l\mid D_0$.
When $l\mid n$, it follows that $\norm\beta\in\ZZ_l$, hence
$\beta\in\Om_K\otimes\ZZ_l$ (since $l$ is ramified).
Since $\norm(\alpha+\beta j)=\norm\alpha+q\norm\beta\in\ZZ$, it
follows that $\alpha\in\Om_K\otimes\ZZ_l$, and the condition at $l$ in
the previous lemma is trivially satisfied for all $\id{d}$.
\par
If $l\nmid n$, then neither $\alpha$ nor $\beta$ are in
$\Om_K\otimes\ZZ_l$, but $l\norm\alpha\in\ZZ_l$ and
$l\norm\beta\in\ZZ_l$.
We claim that this implies $\alpha+\bar\alpha\in\ZZ_l$ and
$\beta+\bar\beta\in\ZZ_l$.
In fact, $l\alpha\in\Om_K\otimes\ZZ_l$, hence
$D\mid\Delta(l\alpha)=(l\trace(\alpha))^2-4l(l\norm\alpha)$.
Since $l\mid D$, it follows that $(l\Tr\alpha)^2\in l\ZZ_l$,
thus $\Tr\alpha\in\ZZ_l$.

Then, the condition
$\norm\alpha+q\norm\beta\in\Om_K\otimes\ZZ_l$
is equivalent to $\alpha^2-q^2\beta^2\in\Om_K\otimes\ZZ_l$
(here we have used that $q\equiv -q^2\pmod{l}$).
Therefore, either $\alpha-q\beta\in\Om_K\otimes\ZZ_l$ or
$\alpha+q\beta\in\Om_K\otimes\ZZ_l$, but not both.
Therefore, the condition at $l$ in the previous lemma is satisfied
for exactly half of the possible $\id{d}$. Namely, when
$\alpha-q\beta\in\Om_K\otimes\ZZ_l$, the condition holds for all
$\id{d}$ with $l\nmid\normid{d}$, and when
$\alpha+q\beta\in\Om_K\otimes\ZZ_l$, the condition holds for all
$\id{d}$ with $l\mid\normid{d}$.

This implies the lemma, since the conditions on $\id{d}$ for each
$l\nmid n$ are independent.
\end{proof}

Let $\id{a}$ and $\id{b}$ be ideals of $\Om_K$ of norm prime to $D$ as
in the lemma, and consider the map 
\[
\Psi_\id{b}:\Hom(\Ot\id{b},\Ot\id{a}\id{b}) \rightarrow \Ix(\Om_K)\times \Ix(\Om_K)
\]
given by
\[
u = \alpha + \beta j \mapsto (\alpha \id{D}_0\id{a}^{-1},\beta
\id{D}_0\id{q} \id{p}^{-1} \id{b} \bar{\id{b}}^{-1} \bar{\id{a}}^{-1}).
\]
Note that $\Psi_\id{b}$ is well defined by Lemma~\ref{lemma:homab} (i.e. it maps to
a pair of \emph{integral} lattices). Furthermore, its
image is contained in
\[
\Lambda :=
\set{(\id{L}_1,\id{L}_2) \st
\id{L}_1 \sim \id{p}\id{a}^{-1},\;
\gen{L_2} = \gen{aq}
},
\]
where by abuse of notation we allow $\id{L}_1$ or
$\id{L}_2$ to be the zero ideal.

Moreover, if
$\Psi_\id{b}(\alpha+j\beta)=(\id{L}_1,\id{L}_2)$
then clearly
\[
\normid{L}_1 = \frac{\norm\alpha\,\abs{D_0}}{\normid{a}},
\qquad
\normid{L}_2 = \frac{\norm\beta\,\abs{D_0} q}{p\normid{a}},
\]
and so
\[
\frac{\normid{L}_1+p\normid{L}_2}{\abs{D_0}}
=
\frac{\norm\alpha+q\norm\beta}{\normid{a}}
=
\frac{\norm(\alpha+\beta j)}{\normid{a}}.
\]
Therefore, if we set
\[
   \Lambda[m] :=
\set{(\id{L}_1,\id{L}_2)\in\Lambda \st
\normid{L}_1+p\normid{L}_2 = m\abs{D_0}},
\]
the maps $\Psi_\id{b}$ restrict to
\[
\Psi_\id{b}:\Hom(\Ot\id{b},\Ot\id{a}\id{b})[m] \rightarrow \Lambda[m].
\]

\begin{lemma} \label{lemma:numL1L2}
    The number of pairs
    $(\id{L}_1,\id{L}_2)\in\Lambda[m]$ with $\normid{L}_2=n$
    is
    \[
\begin{cases}
        1 & \text{if $m=n=0$,} \\
        r_\id{a}(m) & \text{if $m>0$, $n=0$,} \\
        R_{\gen{aq}}(n) & \text{if $m\abs{D}=p^2n\neq 0$,} \\
        r_\id{a}(m\abs{D}-p^2n)
          R_{\gen{aq}}(n) & \text{if $m\abs{D}>p^2n\neq 0$.}
\end{cases}
    \]
\end{lemma}
\begin{proof}
    Since $\normid{L}_2=n$, then the number of choices for this ideal
    is $R_{\gen{aq}}(n)$ if $n\neq0$, and $1$ otherwise.
    Similarly, $\normid{L}_1=m\abs{D_0}-pn$, and the number of choices
    for $\id{L}_1$ is either $r_{\id{pa}^{-1}}(m\abs{D_0}-pn)$ (for
    $m\abs{D}>p^2n$) or $1$ otherwise.
\par
    The result follows by noting that
    $r_{\id{pa}^{-1}}(m\abs{D_0}-pn)=r_\id{a}(m\abs{D}-p^2n)$, because
    $\id{p}$ is ramified; and when $n=0$, since $\sqrt{D}\in\Om_K$, this
    number is just $r_\id{a}(m\abs{D})=r_\id{a}(m)$.
\end{proof}

\begin{lemma} \label{lemma:numPsib}
    Let $(\id{L}_1,\id{L}_2)\in\Lambda[m]$, with
    $\normid{L}_2=n$.
    Then
    \[
    \sum_{\id{b} \in \I(\Om_K)}
    \# \Psi_{\id{b}}^{-1} (\id{L}_1,\id{L}_2)
    = 
    \begin{cases}
        h_D & \text{if $m=n=0$} \\
        2u_Dh_D & \text{if $m>0$, $n=0$,} \\
        2u_D\delta(n) & \text{if $m\abs{D}=p^2n\neq 0$,} \\
        4u_D^2\delta(n) & \text{if $m\abs{D}>p^2n\neq 0$.}
    \end{cases}
    \]
\end{lemma}
\begin{proof}
    First note that $\#\Psi_\id{b}^{-1}(\id{L}_1,\id{L}_2)$ depends
    only on the class of $\id{b}$, since the following diagram
    commutes
\[
\xymatrix@R=0.3ex@C=3em{
\Hom(\Ot\id{b},\Ot\id{a}\id{b})
\ar[drr]^{\Psi_\id{b}}
\ar[dd]_{\text{conjugation by $\gamma$}} \\
         & & \Lambda \\
\Hom(\Ot\id{b}\gamma,\Ot\id{a}\id{b}\gamma)
\ar[urr]_{\Psi_{\id{b}\gamma}}
}
\]
for any $\gamma\in\Om_K$.
\par
Suppose $n\neq 0$.
If $\id{b}^2\id{aq}\nsim\id{L}_2$, then
$\Psi_\id{b}^{-1}(\id{L}_1,\id{L}_2)=\emptyset$.
Fix an ideal $\id{b}_0$ such that
$\id{b}_0^2\id{aq}\sim\id{L}_2$. Then the set of classes $\id{b}$ such that
$\id{b}^2\id{aq}\sim\id{L}_2$ equals $\set{\id{b}_0\id{d}\st
\id{d} \in \calD}$, hence
\[
    \sum_{\id{b} \in \I(\Om_K)}
    \# \Psi_{\id{b}}^{-1} (\id{L}_1,\id{L}_2)
    = 
    \sum_{\id{d}}
    \# \Psi_{\id{b}_0\id{d}}^{-1} (\id{L}_1,\id{L}_2).
\]
\par
Let $\alpha$ be a generator of $\id{L}_1\id{D}_0^{-1}\id{a}$ (which is
principal), and let $\beta$ be a generator of
$\id{L}_2\id{D}_0^{-1}\bar{\id{a}}
 \id{q}^{-1}\id{p}\id{b}_0^{-1}\bar{\id{b}_0}$ (which is also
 principal by the choice of $\id{b}_0$).
Given such a pair $(\alpha,\beta)$, we have
\[
\#\set{\id{d} \st
\alpha+\beta j\in \Psi_{\id{b}_0\id{d}}^{-1}(\id{L}_1,\id{L}_2) }
= \delta(n)
\]
by Lemma~\ref{lemma:sumd}.
\par
Now suppose $n=0$: in this case $\id{L}_2=0$ and it follows from
Lemma~\ref{lemma:sumd} that
$\#\Psi_{\id{b}}^{-1}(\id{L}_1,\id{L}_2)=1$ for all $\id{b}$, thus
\[
    \sum_{\id{b} \in \I(\Om_K)}
    \# \Psi_{\id{b}}^{-1} (\id{L}_1,\id{L}_2)
    =
    h_D.
\]
\par
The statement follows by counting the number of choices for $\alpha$
and $\beta$, which can be $2u_D$ or $1$ in each case (when the norm is
non-zero or zero, respectively).
\end{proof}

The following formula extends \cite[Proposition 10.8]{Gross} to level $p^2$.

\begin{theorem}\label{thm:heights}
Let $\id{a}$ be an ideal for $\Om_K$.
Then
\begin{multline}
\sum_{\id{b} \in \I(\Om_K)}\<\Ot\id{b},\t_{m} \Ot\id{ab}>
=
u_Dh_D r_{\id{a}}(m) \\
+ 2 u_D^2\sum_{n=1}^{\abs{D}m/p^2}
\delta(n)\, r_{\id{a}}(m|D|-p^2n) R_{\gen{aq}}(n),
\label{eq:formula:B}
\end{multline}
where $\delta(n):=2^t$, with $t$ being the number of prime factors of
$\gcd(n,D_0)$.
\end{theorem}
\begin{proof}
    We have
    \[
    \begin{split}
\sum_{\id{b} \in \I(\Om_K)}
\<\Ot\id{b},\t_{m} \Ot\id{ab}>
& =
\frac{1}{2}
\sum_{\id{b} \in \I(\Om_K)}
\#\Hom(\Ot\id{b},\Ot\id{ab})[m] \\
& =
\frac{1}{2}
\sum_{\id{b} \in \I(\Om_K)}
\sum_{(\id{L}_1,\id{L}_2)\in\Lambda[m]}
\#\Psi_\id{b}^{-1}(\id{L}_1,\id{L}_2).
\end{split}
    \]
In order to evaluate this we split the inner sum by the
norm of the ideal $\id{L}_2$. This is suitable to apply
Lemmas~\ref{lemma:numL1L2} and \ref{lemma:numPsib}, which together
give
\begin{multline*}
\frac{1}{2}
\sum_{\id{b} \in \I(\Om_K)}
\sum_{\substack{(\id{L}_1,\id{L}_2)\in\Lambda[m]\\\normid{L}_2=n}}
    \# \Psi_{\id{b}}^{-1} (\id{L}_1,\id{L}_2)
    \\
    =
    \begin{cases}
        u_D h_D r_\id{a}(m) & \text{if $n=0$,} \\
        2u_D^2\,\delta(n)\,r_\id{a}(m\abs{D}-p^2n)\,R_{\gen{aq}}(n) &
        \text{if $n\neq 0$.}
    \end{cases}
\end{multline*}
Note that the four different cases of the lemmas become just two cases
by use of the convention
$r_\id{a}(0)=\frac{1}{2u_D}$.
\par
The statement follows by adding this expression over $n\geq 0$.
\end{proof}

\subsection{Special points and central values of $L$-series}

Let $f$ be a cusp form in $S_2^\new(\Gamma_0(N))$.
In the appendix we recall the definition
of an $L$-series $L_{\sA}(f,s)$, a Rankin convolution of
$L(f,s)$ and a partial zeta function associated to
an ideal class $\sA$ of $\QQ(\sqrt{D})$.
These $L$-series are interesting due to their relation to the $L$-series of $f$ and its twists;
for instance we have the factorization
\[
\sum_{\sA} L_{\sA}(f,s) = L(f,s)\, L(f, D, s),
\]
where the sum is over all ideal classes of $\QQ(\sqrt{D})$.
The main result of the appendix is the following generalization of
\cite[(4.4) p.283]{GZ} regarding the central values of this $L$-series.
\begin{rankinthm}
Let $D<0$ be an odd fundamental discriminant, $\sA$ be an ideal in
$\Q[\sqrt{D}]$ and $f(z)$ be a cusp form in
$S_2^\new(\Gamma_0(N))$. Then, 
    \[
    L_{\sA}(f,1) = \frac{8\pi^2}{\sqrt{\abs{D}}} \<f, g_{\sA}>,
    \]
with $g_{\sA} = g_{\sA}^{(N)} = \sum b_{\sA}(m) q^m$,
where
\begin{multline}
  b_{\sA}(m) :=
  \frac{1-\varepsilon_D(N\eta)}{2}\cdot \frac{h(D)}{u_D} r_{\sA}(m)
  \\
  +
  \sum_{\gen{q}\in Q}
  \sum_{n=1}^{{\abs{D}m}/{N}}
  \delta(n) r_{\sA}(m\abs{D} - nN)
  R_{\gen{\sA\id{q}}}(n),
  \label{eq:formula:A}
\end{multline}
where the first sum is over the set of genera
\[
Q := \set{ \gen{q} \st \normid{q}\equiv -N\pmod{D_0}},
\]
and where $\delta(n):=2^t$, with $t$ the number of prime factors of
$\gcd(n,D_0)$.
\end{rankinthm}

Comparing the right hand side of~\eqref{eq:formula:B} in Theorem~\ref{thm:heights}
with the formula~\eqref{eq:formula:A} for the Fourier coefficients of the form
$g_{\sA}^{(p^2)}$ in Theorem~\ref{thm:rankin}, we obtain an explicit
formula for $g_{\sA}^{(p^2)}$ in terms of special points.

\begin{corollary} \label{cor:g_A}
On the above notation,
\begin{equation}
  g_{\sA}^{(p^2)} = 
    \frac{1}{2u_D^2}
\sum_{\id{b} \in \I(\Om_K)}
\Hecke(\Ot\id{b},\Ot\id{ab}),
\end{equation}
\end{corollary}
\begin{proof}
    In the formula for $b_{\sA}(m)$ of Theorem~\ref{thm:rankin}, 
    note that $\varepsilon_D(N\eta)=0$ (since $N=p^2$ and $p\mid D$
    in our case), and the set $Q$ consists of a unique element $\gen{q}$.
    Then the statement follows immediately from Theorem~\ref{thm:heights}.
\end{proof}

Assume now $f$ is a normalized eigenform in $S_2^\new(p^2)$.
Fix a character $\varphi$ of $\I(\Om_K)$, and define
\[
L_\varphi(f,s) := \sum_{\sA} \varphi(\sA) L_{\sA}(f,s).
\]
Consider $\vec{c}_{\varphi} = \sum_{\id{a}} \varphi^{-1}(\id{a}) \,{\Ot{\id{a}}} \in
\M(\Ot)$, and denote $\vec{c}_{f,\varphi}$ its projection to the $f$-isotypical
component of $\M(\Ot)$.
\begin{proposition}
\label{prop:centralvalue}
\[
    L_\varphi(f,1) = \frac{4\pi^2}{u_D^2}\cdot\frac{\<f,f>}{\sqrt{\abs{D}}}
    \, \<\vec{c}_{f,\varphi},\vec{c}_{f,\varphi}>.
\]
\end{proposition}
\begin{proof}
    The proof is similar to Proposition 11.2 of \cite{Gross},
    given Theorem~\ref{thm:rankin} and Corollary~\ref{cor:g_A}.
In our case, the $f$-isotypical component in $\M(\Ot)$ may have
dimension $2$; however, since $f$ is new, the $f$-isotypical component
of $S_2^\new(p^2)$ has dimension $1$, and the same reasoning
as given by Gross still applies.
\end{proof}

\begin{remark}
    If the $f$-isotypical component $\M(\Ot)^f$ is zero, the
    proposition implies that $L_\varphi(f,1)=0$ for all characters
    $\varphi$. Equivalently,
    \[
       L_{\sA}(f,1) = 0
    \]
    for all ideal classes $\sA$, and for all discriminants, whenever
    $f$ is a twist of a form of level $p$ by a non-quadratic
    character.
\end{remark}

\section{Proof of the Main Theorem}
\label{sec:main}

In this section we want to relate the central value of the $L$-series 
$L_\varphi(f,1)$ to coefficients of half-integral weight modular forms.
We will assume from now on that $\varphi=1_D$. This case of the Rankin
convolution $L$-series is related to our main formula because of the
factorization
\[
   L_{1_D}(f,s) = L(f, s) \, L(f, D, s).
\]

We will start with the case of odd discriminants D, which follows from
the results in Section~\ref{sec:specialpoints}.
The case of even discriminants could
be proved by a similar calculation, but we avoid the technical
difficulties of this case by resorting to a theorem of Waldspurger.
This step is done in the proof of Theorem~\ref{thm:main}; until then
we will assume that the discriminant D is odd, just so that we can use
the results in previous sections.

Let $\calP$ denote the set of norm $p$
bilateral $\Ot$-ideals. If $\id{b} \in \M(\Ot)$, $\id{p} \in \calP$,
and $D$ is a negative fundamental discriminant, with $D = -pd$, the coefficient 
of $q^{d}$ in the $q$-expansion of $\Theta_{\id{p}}(\cls{b})$ is 
\begin{equation}
\label{eq:splitting}
c_{d,\id{p}}(\id{b}) = \frac{1}{2} \#\set{x \in \ZZ +\id{b}^{-1} \id{p} \id{b}/\ZZ \st
\Delta x = D}.
\end{equation}
Let $\vec{c}_{d,\id{p}}:= \sum_{\cls{b}} c_{d,\id{p}}(\id{b}) \cls{b}$. Then
if $\vec{e} \in \M(\Ot)$, the coefficient of $q^d$ in the
$q$-expansion of $\Theta_{\id{p}}(\vec{e})$ is
$\<\vec{c}_{d,\id{p}},\vec{e}>$.

We want to give an ad\`elic description of this set.
Let $\omega_D\in \Om_K$ an element of discriminant $D$; adding an integer we
will assume that $\trace \omega_D\equiv 0 \pmod{p}$.
It's easy to check that then $\Om_K=\<1,\omega_D>$ and
$\id{p}_K=\<p,\omega_D>$.
Moreover,
\[
4\,\norm\omega_D
=(\trace\omega_D)^2 - D
\equiv - D \pmod{p^2}.
\]
Fix an embedding $i: K\hookrightarrow B$, and let
$\vec{x}=i(\omega_D)$.
Once such embedding is fixed, by
Proposition~\ref{prop:specialpointscorrespondence}
the special points of discriminant $D$
correspond to some elements in the double coset
$\Ix(\Ot)/ K^\times$.
Explicitly, let $x \in \ZZ + \id{b}^{-1}
\id{p} \id{b}$ of discriminant $D$; adding an integer we may assume that
$\trace x = \trace\vec{x}$.
Hence $\norm x=\norm\vec{x}$ as well, and so there exists $\alpha \in B^\times/K^\times$ with 
\[
\alpha^{-1} x \alpha = \vec{x}.
\]
The correspondence associates to $x$ the $\Om_K$-point
$\id{b}\alpha$. The condition $x \in \ZZ + \id{b}^{-1}\id{p}
\id{b}$ translates to the condition $\vec{x} \in
\ZZ + (\id{b}\alpha)^{-1} \id{p} (\id{b}\alpha)$.

For $\id{p} \in \calP$, define
\[
\calC_\id{p} := \set{\id{a} \in \Ix(\Ot) \st (\id{a},i) \text{
is a special point for $\Om_K$ and } \vec{x} \in \id{a}^{-1}\id{p}
\id{a}}.
\]
Recall from Section~\ref{sec:bilateral} that $(\id{a},i)$ is a special point for
$\Om_K$ if $\Or(\id{a}) \cap K = \Om_K$.

\begin{lemma} $\calC_\id{p}$ is closed under the action of
    $\widehat{K}^\times$ by right multiplication, i.e. $\calC_\id{p}
    \widehat{K}^\times = \calC_\id{p}$.
    \label{lemma:CP1}
\end{lemma}
\begin{proof}
    Let $\id{a} \in \calC_\id{p}$ and $\hat{\alpha} \in
    \widehat{K}^\times$. By Proposition~\ref{quadraticaction},
    $\id{a}\hat{\alpha}$ is an $\Om_K$-point. Since
    $\vec{x} \in \id{a}^{-1} \id{p} \id{a}$, 
    \[
    \hat{\alpha}^{-1} \vec{x} \hat{\alpha} \in (\id{a}\hat{\alpha})^{-1} 
    \id{p} (\id{a} \hat{\alpha}).
    \]
    But $\hat{\alpha}^{-1} \vec{x} \hat{\alpha} = \vec{x}$, because all elements
    are in $\widehat{K}$, which is commutative.
    Then $\id{a}\hat{\alpha} \in \calC_\id{p}$ as
    claimed.
\end{proof}

\begin{lemma} $\calC_\id{p}$ is closed under the action of $\Wt$
    by left multiplication.
    \label{lemma:CP2}
\end{lemma}
\begin{proof} Recall from Section~\ref{sec:bilateral} that the bilateral ideals act on
    the $\Om_K$-points by left multiplication.
    Let $\id{a} \in \calC_\id{p}$;
    then $\vec{x} \in \id{a}^{-1}\id{p} \id{a}$, and
    we must show that $\Wt\id{a}\in\calC_\id{p}$.
    But $\Wt$ is the order two rotation of the group of
    bilateral $\Ot$-ideals (which is a dihedral group), hence it
    commutes with $\id{p}$ for any $\id{p} \in \calP$. Thus, 
\[
\vec{x} \in \id{a}^{-1}\id{p} \id{a} =
(\Wt\id{a})^{-1} \id{p} (\Wt \id{a}),
\]
hence $\Wt\id{a}\in\calC_\id{p}$ as claimed.
\end{proof}

    The last two lemmas imply that $\calC_\id{p}$ is closed under the
    action of $\{1,\Wt\} \times \I(\Om_K)$. 
Moreover, $c_{d,\id{p}}(\id{b}) = \frac{1}{2}\#
\set{\id{a} \in \calC_\id{p} / K^\times \st \id{a} \sim \id{b} }$.

Let $\O$ be an order in the same genus of $\Ot$ such that
$i(\Om_K)\subseteq\O$. In particular, $\vec{x} \in \O$.
Such an order $R$ exists because $p\mid D$.
Indeed, $i(\Om_K)$ is contained in some maximal order $\O_0$ of $B$; but the
condition $p\mid D$ implies that $p\mid\normx{x}$ for all $x\in
i(\Om_K)$, hence $i(\Om_K)$ is actually contained in the unique order
of index $p$ in $\O_0$.

\begin{lemma}
    \label{lemma:alpha}
    Assume that $\chi(\id{p}) = \kro{d}{p}$.
    Then, there is an $\hat\alpha\in\widehat{B}^\times$ such that
    \begin{enumerate}
        \item 
            $\hat\alpha^{-1}\Ot\hat\alpha = \O$,
        \item 
            $\hat\alpha^{-1}\id{p}\hat\alpha = \O \,i(\id{p}_K)$.
    \end{enumerate}
    Moreover, $\hat\alpha$ modulo multiplication by $\QQ^\times$,
    is unique up to left multiplication by the group generated by
    $\Wt$ and $W_\id{p}$.
\end{lemma}
\begin{proof}
    Since $\O$ is in the same genus as $\Ot$, there exists
    $\hat\alpha$ satisfying the first condition. Two such elements
    differ by a bilateral $\Ot$-ideal.
    \par
    We claim that the left $\O$-ideal $\O \,i(\id{p}_K)$ is a
    bilateral ideal. It is enough to prove that all the localizations
    are bilateral ideals. Since $\O \,i(\Om_K) = \O$, and $(\Om_K)_q =
    (\id{p}_K)_q$ at all primes $q \neq p$, the localizations at
    primes $q \neq p$ give bilateral ideals. At the ramified prime
    $p$, there is a unique maximal order, and a unique order of index
    $p$ in it, hence the localization at $p$ must also give a
    bilateral ideal.
    \par
    Recall that the group of bilateral 
    $\Ot$-ideals modulo $\Q^\times$ is a dihedral group.
    We claim that $\chi(\O\, i(\id{p}_K))=\kro{d}{p}$. Indeed,
    $\vec{x}\in\O\,i(\id{p}_K)$ and
    $\norm{\vec{x}}=\norm{\omega_D}\equiv \frac{-D}{4}\pmod{p^2}$,
    hence $\kro{\norm{\vec{x}}/p}{p}=\kro{-D/(4p)}{p}=\kro{d}{p}$.
    Then by assumption
    $\chi(\O\, i(\id{p}_K))=\chi(\id{p})$,
    hence $\hat\alpha^{-1}\id{p}\hat\alpha$ and $\O\,i(\id{p}_K)$,
    being bilateral,
    are in the same orbit by conjugation in this dihedral group.
    Thus the second condition holds by changing $\hat\alpha$
    accordingly.
    \par
    The last statement follows from the fact that the stabilizer of
    $\id{p}$ in the dihedral group is the group generated by $\Wt$ and
    $W_\id{p}$.
\end{proof}

Note that, for an ideal class $\cls{a}$ in $\I(\Om_K)$, the
$\Ot$-ideal $\hat\alpha\O \, i(\id{a})$ with $\hat\alpha$ as in the
lemma is an $\Om_K$-point, since
$\vec{x} \in i(\id{p}_K)$. Then we define
\[
   \calC_{D,\id{p}} :=
   \set{ \vphantom{\Wt} \hat\alpha\,\O\, i(\id{a}) \st \cls{a}\in\I(\Om_K)}
   \cup
   \set{ \Wt\hat\alpha\,\O\, i(\id{a}) \st \cls{a}\in\I(\Om_K)}.
\]

\begin{lemma}\
    \begin{enumerate}
        \item $W_\id{p}\,\calC_{D,\id{p}} = \calC_{D,\id{p}}$.
        \item $\calC_{D,\id{p}}$ is independent of the choice of $\hat\alpha$. 
    \end{enumerate}
\end{lemma}
\begin{proof}
The first statement follows from
$$\id{p}\hat\alpha R\, i(\id{a}) =
\hat\alpha\O\, i(\id{p}_K)\, i(\id{a}) =
\hat\alpha\O\, i(\id{p}_K\,\id{a}),$$
since multiplication by $\id{p}_K$ is a permutation of $\I(\Om_K)$.
For the second statement note that clearly
$\Wt\,\calC_{D,\id{p}}=\calC_{D,\id{p}}$, and use the final statement of
the previous lemma.
\end{proof}

\begin{theorem} With the previous notation,
\[
\calC_{D,\id{p}} = \calC_\id{p} /K^\times.
\]
    \label{thm:specialpointsrelation}
    \end{theorem}
    \begin{proof}

        Let $\hat\alpha$ be as in Lemma~\ref{lemma:alpha},
        so that $\hat\alpha\O\in \calC_{D,\id{p}}$.
        Then $\hat\alpha\O=\Ot\hat\alpha$ by the first condition in the
        lemma, and the second condition implies that it is in
        $\calC_\id{p}$.

        Then, by Lemma~\ref{lemma:CP1} and Lemma~\ref{lemma:CP2},
        $\calC_{D,\id{p}}\subseteq \calC_{\id{p}} / K^\times$.
       The
       theorem follows from the fact that $
       \calC_\id{p}/K^\times$ has exactly $2 h_D$ elements, where
       $h_D$ is the class number of $\Om_K$. This will be proved in
       the following lemmas.

    \begin{lemma} We have
        \label{lemma:specialpointsdisjointunion}
        \[
        \{\Om_K \text{-points for }\Ot\} = \bigcup_{\id{p} \in \calP} \calC_\id{p}/K^\times.
        \]
        
        Furthermore, the union is disjoint.
    \end{lemma}
    
    \begin{proof} The fact that the set of $\Om_K$-points are the
        union over the norm $p$ bilateral $\Ot$-ideals follows
  from the fact that there is a bijection between bilateral ideals of
  norm $p$ and orders of index $p$ in $\Ot$ (given by $\id{p} \leftrightarrow
  \ZZ +\id{p}$) and the union of such orders is $\Ot$. The second claim
  comes from the fact that the intersection of two different index $p$ 
  suborders of $\Ot$ gives $\ZZ + p \Om$, hence the discriminant of such
  elements is divisible by $p^2$.
\end{proof}

\begin{lemma} If $\chi(\id{p}) \neq \kro{d}{p}$ then
  $\calC_{\id{p}} = \emptyset$. 
  \label{lemma:specialpointemptyset}
\end{lemma}
\begin{proof} If $\calC_\id{p} \neq \emptyset$, there exists
    $\hat{\alpha}$ such that $\vec{x} \in \hat{\alpha}^{-1} \id{p}
    \hat{\alpha}$. By definition, $\chi(\hat{\alpha}^{-1}\id{p}
    \hat{\alpha})$ is computed by choosing an element in this ideal of
    norm divisible by $p$ but not by $p^2$. Since $\vec{x}$ is such an
    element, 
    \[
    \chi(\id{p})
    = \chi(\hat{\alpha}^{-1}\id{p} \hat{\alpha})
    = \kro{\norm{\vec{x}}/p}{p}
    = \kro{-D/p}{p}
    = \kro{d}{p}
    .
    \]
\end{proof}
\begin{proposition} If $\id{p}_1,\id{p}_2 \in \calP$ with $\chi (\id{p}_1) =
\chi(\id{p}_2)$ then
$\calC_{\id{p}_1}$ and $\calC_{\id{p}_2}$ have the same number of elements.
\label{prop:specialpointsamecharacter}
\end{proposition}

\begin{proof} If $\chi (\id{p}_1) = \chi(\id{p}_2)$ there exists a rotation $\id{o}$ in
  the group of bilateral ideals such that $\id{o}^{-1} \id{p}_2 \id{o} =
  \id{p}_1$. Then $\id{o} \calC_{\id{p}_1} = \calC_{\id{p}_2}$, where the
  action is given by left multiplication.\hbox to 2ex{}
\end{proof}

\begin{lemma} The number of $\Om_K$-points for $\Om$ is $h_D$.
\label{lemma:numberofspecialpoints}
\end{lemma}

\begin{proof} The set of $\Om_K$-points is an homogeneous space for
  $\Bil(\Om)/\Q^\times \times \I(\Om_K)$. The results follows from
  the fact that $\Bil(\Om)/\Q^\times$ has two elements, and the norm
  $p$ ideal in $\Bil(\Om)$  acts as the ideal $\id{p}_K\in\Ix(\Om_K)$.
\end{proof}

\begin{proposition} The number of $\Om_K$-points for $\Ot$ is $(p+1) h_D$.
\end{proposition}
\begin{proof}
See Theorem 2.7 and Theorem 4.8 of \cite{Pizer2}.
\end{proof}

The last proposition asserts that the total number of $\Om_K$-points for
$\Ot$ is $(p+1) h_D$. By Lemma~\ref{lemma:specialpointsdisjointunion},
this equals $\sum_{\id{p} \in \calP} \#
\calC_{\id{p}}/K^\times$. By Lemma~\ref{lemma:specialpointemptyset}, half of this
numbers are zero and by
Proposition~\ref{prop:specialpointsamecharacter} all the non-empty
sets have the same number of elements. This implies that
\[
\#\calC_\id{p}/K^\times =2h_D,
\]
which finishes the proof of Theorem~\ref{thm:specialpointsrelation}.
\end{proof}

In particular, $2\, \vec{c}_{d,\id{p}}=\sum_{\id{b} \in \calC_{D,\id{p}}}
\cls{b}$ in $\M(\Ot)$. Since $W_{\id{p}}$ commutes with $\HeckeRing_0$ and acts trivially on
$\vec{c}_{d,\id{p}}$, then the projection $\vec{c}_{f,\id{p}}$ to the
$f$-isotypical component is a vector on which $W_\id{p}$ acts trivially.

Assume there is a non-zero eigenvector $\vec{e}_f\in\M(\Ot)^f$ with
$W_\id{p}\,\vec{e}_f = \vec{e}_f$, as otherwise $\vec{c}_{f,\id{p}}=0$ and
$L(f,1)=0$ by Proposition~\ref{prop:centralvalue}.  When
$\dim\M(\Ot)^f=2$, this vector always exists by
Theorem~\ref{thm:multone} (multiplicity one).
In any case $\vec{e}_f$ is unique up to a constant, and therefore
\[
\vec{c}_{f,\id{p}} = \frac{\<\vec{c}_{d,\id{p}},\vec{e}_f>}
{\<\vec{e}_f,\vec{e}_f>}\,
\vec{e}_f,
\]
and
\begin{equation}
\label{SP}
\<\vec{c}_{f,\id{p}},\vec{c}_{f,\id{p}}> =
\frac{\abs{\<\vec{c}_{d,\id{p}},\vec{e}_f>}^2}
{\<\vec{e}_f,\vec{e}_f>}.
\end{equation}

\begin{theorem}[Main Theorem] 
\label{thm:main}
Let $f$ be a new eigenform of weight $2$, level $p^2$ with $p>2$ an odd prime.
Fix a norm $p$ bilateral $\Ot$-ideal $\id{p}$,
and let $\vec{e}_f$ be an eigenvector in
the $f$-isotypical component of
$\M(\Ot)$ such that $W_{\id{p}}(\vec{e}_f) = \vec{e}_f$.
\par
If $d$ is an integer such that $D=-pd<0$ is a fundamental discriminant, 
and such that $\kro{d}{p}=\chi(\id{p})$, then
\[
L(f,1)\,L(f,D,1)
=
4 \pi^2\,\frac{\<f,f>}{\<\vec{e}_f,\vec{e}_f>}
                        \,\frac{c_d^2}{\sqrt{pd}},
\]
where the $c_d$ are the Fourier coefficients of
$\Theta_{\id{p}}(\vec{e}_f) = \sum_{d \ge 1}c_d\,q^d$.
\end{theorem}

\begin{proof} $\bullet$ \emph{First case: odd discriminants.}

By Proposition~\ref{prop:centralvalue},
\[
    L(f,1)\, L(f,D,1) = L_{1_D}(f,1) = \frac{4\pi^2}{u_D^2}\cdot\frac{\<f,f>}{\sqrt{\abs{D}}}
    \, \<\vec{c}_{f},\vec{c}_{f}>,
\]
where $\vec{c}_f$ is the projection of $\vec{c}_1=\sum_{\id{a}} R
\id{a}$ to the $f$-isotypical component in $\M(\Ot)$. Since
$W_{\id{p}}$ acts trivially in $\vec{c}_1$, it acts trivially in its
projection $\vec{c}_f$ so we just need to project it to the eigenspace
where $W_{\id{p}}$ acts trivially.

Since $2\, \vec{c}_{d,\id{p}}=\sum_{\id{b} \in \calC_{D,\id{p}}}
\cls{b}$, from the definition of $\calC_{D,\id{p}}$ we get
\[
\<\vec{c}_{d,\id{p}},\vec{e}_f> = \frac{1}{2}\left(\<\vec{c}_1,\vec{e}_f> + \<\Wt \vec{c}_1,\vec{e}_f>\right).
\]

If the operator $\Wt$ acts as $-1$ in
$\vec{e}_f$, then $c_d = \<\vec{c}_{d,\id{p}},\vec{e}_f>=0$ and the
left hand side of the main formula also vanishes (since the sign of
the functional equation is
$-1$ in this case). Otherwise, $\Wt$ acts trivially in $\vec{e}_f$ and
so $\<\vec{c}_{d,\id{p}},\vec{e}_f>=\<\vec{c}_1,\vec{e}_f>$, hence
$\vec{c}_{f,\id{p}}=\vec{c}_f$.
  Therefore
\[
\<\vec{c}_f,\vec{c}_f> =
\<\vec{c}_{f,\id{p}},\vec{c}_{f,\id{p}}> =
\frac{\abs{\<\vec{c}_{d,\id{p}},\vec{e}_f>}^2}
{\<\vec{e}_f,\vec{e}_f>} = \frac{{c_d}^2}{\<\vec{e}_f,\vec{e}_f>}.
\]

\noindent $\bullet$ \emph{Second case: even discriminants.} The
modular form $\Theta_\id{p}(\vec{e}_f)$ lies in the space $S_{3/2}(4p^2,
\charp)$ and maps to $f$ via the Shimura correspondence. In this
situation, Corollary $2$ of \cite{Waldspurger} states:

\begin{theorem}[Waldspurger] Let $pd_1, pd_2 \in \NN$ be square-free
  integers, and suppose that $d_1/d_2 \in {\Q_q^\times}^2$ for $q=p$
  and $q=2$. Then one has the equality
\[
c_{d_1}^2 \, L(f,-pd_2^\dag,1)\, \sqrt{pd_2} = c_{d_2}^2 \,
L(f,-pd_1^\dag,1)\, \sqrt{pd_1},
\]
where $-pd_i^\dag$ is the discriminant of the quadratic field
$\Q[\sqrt{-pd_i}]$. 
\end{theorem}

For the particular case of $\Theta_\id{p}(\vec{e}_f)$, one
actually has

\begin{proposition} Let $-p d_1$ and $-pd_2$ be fundamental discriminants and
  suppose that $d_1/d_2 \in {\Q_p^\times}^2$. Then one has the equality
\label{prop:walds}
\[
c_{d_1}^2 \, L(f,-pd_1,1)\, \sqrt{pd_2} = c_{d_2}^2 \,
L(f,-pd_2,1)\, \sqrt{pd_1}.
\]
\end{proposition}
We show that the case of even discriminants of the Main Theorem follows from
Proposition \ref{prop:walds}.
By Theorem $4$ of \cite{Waldspurger91}, there exists an odd
fundamental discriminant $-pd_0$ such that $L(f,-pd_0,1) \neq 0$. The
Main Theorem for odd fundamental discriminants implies that the
coefficient $c_{d_0} \neq 0$. Let $-pd$ be an even fundamental
discriminant, then
\begin{align*}
L(f,-pd,1)L(f,1) & = \frac{L(f,-pd,1)}{L(f,-pd_0,1)}\, L(f,-pd_0,1)
\, L(f,1)\\
& = \frac{c_d^2}{\sqrt{pd}} \, \frac{\sqrt{pd_0}}{c_{d_0}^2}\, L(f,-pd_0,1)
\,L(f,1) \\
& = 4 \pi^2 \, \frac{\<f,f>}{\<\vec{e}_f,\vec{e}_f>}\, \frac{c_d^2}{\sqrt{pd}}
\end{align*}
where the second equality follows from Proposition \ref{prop:walds}
and the last one follows from the Main Theorem for odd fundamental
discriminants.
\end{proof}

\bgroup
\def\proofname{Proof of Proposition \ref{prop:walds}}
\begin{proof}
    To prove the
  result, we follow the proof of the Corollary $2$ in
  \cite{Waldspurger}. The same reasoning implies the result once we
  prove that the local factor at $2$ of the weight $3/2$ modular form
  $\Theta_\id{p}(\vec{e}_f)$ is the same for $pd_1$ and $pd_2$. Let
  $\lambda_2$ be the eigenvalue of the Hecke operator $T_2$ acting on
  $f$. Let $\alpha, \alpha'$ denote the roots of the polynomial
  $x^2-\frac{\lambda_2}{\sqrt{2}}x +1$.
\par
The space $S_{3/2}(4p^2, \charp)$ is $4$-dimensional. Generators for
this space are obtained by choosing $2$ local functions at the prime
$p$ and $2$ local functions at the prime $2$. Following the notation
of \cite{Waldspurger} (p.453), define the functions $c_2'[\delta]$ and
$c_2''[\delta]$ on $d$ (where $-pd$ is a fundamental discriminant) by:
\[
c_2'[\delta](d) :=
\begin{cases}
\delta - (2,d)_2 \, \charptwo(2)/\sqrt{2} & \text{ if } 2\nmid d.\\
\delta & \text{ if }2 \mid d,
\end{cases}
\]
where $(*,*)_2$ denotes the Hilbert symbol at $2$, and $\charptwo$ is
the character on $\Q_2^\times$ associated to $\charp$. And the
function
\[
c_2''[\delta](d):=\delta.
\]
\par
Then the set of local functions at the prime $2$ is given by 
\[
\begin{cases}
\{c_2'[\alpha], c_2'[\alpha']\} & \text{ if }\alpha \neq \alpha',\\
\{c_2'[\alpha], c_2''[\alpha]\} & \text{ if }\alpha = \alpha'.
\end{cases}
\]
Moreover, we have $c_2'[\alpha](d)=1$ for $2\parallel d$. But since
the coefficients $c_d$ of $\Theta_\id{p}(\vec{e}_f)$ vanish in this
case, its local function at $2$ (up to a global constant)
must be, in the first case,
\[
c_2'[\alpha] - c_2'[\alpha'].
\]
This function clearly attains the same value at odd and even values
of $d$ (namely $\alpha - \alpha'$).

In the second case, the local function at $2$ must be $c_2''[\alpha]$,
since $c_2''[\alpha](d)=0$ when $2\parallel d$.
This function also clearly attains the same value for odd and even
values of $d$.
The rest of the proof is exactly the same as Waldspurger.
\end{proof}
\egroup

\appendix
\renewcommand{\thesection}{\Alph{section}}
\section{Rankin's Method}

\begin{notation}
If $n,m$ are integers, we write $n \mid m^\infty$ if every prime
factor of $n$ divides $m$. We denote by $\gcd(n,m^\infty)$ the
unique positive integer $M$ that satisfies 
\begin{itemize}
\item $M \mid n$,
\item $M \mid m^\infty$,
\item $\gcd(\frac{n}{M},m)=1$.
\end{itemize}
\end{notation}

    Let $D<0$ be a fundamental discriminant.
    If $\sA$ is an ideal class of $\QQ(\sqrt{D})$,
    we denote $\Theta_{\sA}$ the theta series
    \[
       \Theta_{\sA}(z)
       := \sum_{n=0}^\infty r_{\sA}(n) q^n
       = \frac{1}{2} \sum_{x \in \id{a}} q^{\norm(x)/\norm{\id{a}}},
    \]
    where $\id{a}$ is any ideal in the class $\sA$. It is well known that
    $\Theta_{\sA}$ is a weight 1 modular form of level $\abs{D}$ and
    nebentypus $\varepsilon_D$, where
    $\varepsilon_D:(\Z/D\Z)^\times \rightarrow \CC ^\times$ denotes the
    character $\varepsilon_D(n)=\kro{D}{n}$ of the field
    $\QQ(\sqrt{D})$.

\begin{definition}
Let $f(z)=\sum a(n) q^n$ be a cusp form in $S_2^\new(\Gamma_0(N))$.
Define
$$
L_{\sA}(f,s) :=
\left(
\sum_{(m,N)=1} \frac{\varepsilon_D(m)}{m^{2s-1}}
\right)
\cdot
\left(
\sum_{m=1}^\infty \frac{a(m) r_{\sA}(m)}{m^s}
\right)
$$ which converges for
  $\Re(s) >3/2$.
\end{definition}

\subsection{Rankin's Method}

For each decomposition $D = D_1 D_2$ of $D$ as the product of two
fundamental discriminants, define the Eisenstein series 
\[
    E^{(D_1,D_2)}_s(z) := \frac{1}{2}
    \sum_{\substack{m,n\in\ZZ \\ D_2\mid m}}
    \frac{\varepsilon_{D_1}(m)\,\varepsilon_{D_2}(n)}
    {(mz+n)} \frac{y^s}{|mz+n|^{2s}}
\]
The series $E^{(D_1,D_2)}_s(z)$ is a non-holomorphic weight $1$ modular
form of level $\abs{D}$ and Nebentypus $\varepsilon_{D}$.

Let $\eta=\gcd(N,D)$ and $N_0=N/\eta$.
In \cite{GZ}, they work in the case
$\eta=1$; when $N$ is a perfect square, this restriction makes their
formula for the central value vanish trivially on both sides (see the
remark after Proposition~\ref{prop:sigmaA}). 
\begin{proposition}
    \label{prop:rankin}
    \[
    (4\pi)^{-s} \Gamma(s)  L_{\sA}(f,s) = \<f,
    G_{\bar s-1,\sA}>_{\Gamma_0(N)}
    \]
    where
    \[
    G_{s,\sA}(z) := \trace^{N_0 \abs{D}}_{N}
    \left(\Theta_{\sA}(z) E_{s}^{(1,D)} (N_0z)\right)
    \]
\end{proposition}
\begin{proof} Similar to \cite[(1.2) p. 272]{GZ}.
\end{proof}

\subsection{Computation of the trace}

For $D$ a fundamental discriminant, let 
\[
\kappa(D) :=
\begin{cases}
1 & \text{ if }D>0\\
i & \text{ if }D<0.
\end{cases}
\]

If $D=D_1D_2$ is a decomposition of $D$ as the product of two
fundamental discriminants, $\chi_{D_1,D_2}$ denotes the corresponding
genus character, i.e. for ideals $\sA$ of norm prime to $D$,
$\chi_{D_1,D_2}(\sA) = \varepsilon_{D_1}(\norm \sA) =
\varepsilon_{D_2}(\norm \sA)$.

Recall the usual operator
\[
U_{\abs{D}}(f) := \frac{1}{\abs{D}}
\sum_{j \bmod{\abs{D}}} f\left(\frac{z+j}{\abs{D}}\right)
\]
on spaces of modular forms.

\begin{proposition}
\label{prop:tracecomputed}
Assume $D$ is odd. Then the function $G_{s,\sA}(z)$ defined in the
last proposition is given by
    \[
    G_{s,\sA}(z) = (\sE_s(N_0z)\,\Theta_{\sA}(z))_{|U_{\abs{D}}},
    \]
where 
\begin{equation}
    \sE_s(z) := \sum_{D = D_1D_2}
\frac{\varepsilon_{D_1}(N\eta)\, \chi_{D_1,D_2}(\sA)}
{\kappa(D_1)\,\abs{D_1}^{s+1/2}} E_s^{(D_1,D_2)}(\abs{D_2}z)
\label{eq:sEs}
\end{equation}
The sum is over all decompositions of $D$ as a product of two
fundamental discriminants $D_1$ and $D_2$.
\end{proposition}
\begin{proof}
    As in \cite[(2.4) p.276]{GZ}, we can prove the result with
    \[
    \sE_s(z) = \sum_{\substack{D = D_1D_2\\\gcd(D_1,\eta)=1}}
\frac{\varepsilon_{D_1}(N_0)\, \chi_{D_1,D_2}(\sA)}
{\kappa(D_1)\,\abs{D_1}^{s+1/2}} E_s^{(D_1,D_2)}(\abs{D_2}z)
    \]
    and the statement follows since 
    \[
    \varepsilon_{D_1}(N\eta) =
    \begin{cases}
        \varepsilon_{D_1}(N_0) & \text{when $\gcd(D_1,\eta)=1$,}
        \\
        0 & \text{otherwise.}
    \end{cases}
    \]
\end{proof}

\subsection{Fourier expansions}
From now on we will assume $D$ is odd as in
Proposition~\ref{prop:tracecomputed}. This implies that $\eta$ is odd
and squarefree.

We first give an explicit description of the Fourier
coefficients of the function $\sE_s(z)$ defined in~\eqref{eq:sEs}
above.
\begin{proposition}
We have
\[
\sE_s(z) = \sum_{n\in\ZZ} e_s(n,y) e^{2\pi i nx}.
\]
where the coefficients are given by
\[
e_s(0,y) =
  L(\varepsilon_D,2s+1)
  \, (\abs{D}y)^s
  +
  \frac{\varepsilon_D(N\eta)}{i\sqrt{\abs{D}}}
  \, V_s(0)
  \, L(\varepsilon_D,2s)
  \, (\abs{D} y)^{-s}
\]
if $n=0$ and by
\[
e_s(n,y) =
  \frac{i}{\sqrt{\abs{D}}}
  \,(\abs{D}y)^{-s}
  \, V_s(ny)
  \,\sum_{\substack{d\mid n\\d>0}} \frac{\varepsilon_{\sA}(n,d)}{d^{2s}}
\]
if $n\neq 0$,
with
\[
  \varepsilon_{\sA}(n,d) :=
  \varepsilon_{D_1}(-N\eta\,d)
  \, \varepsilon_{D_2}(n/d)
  \, \chi_{D_1,D_2}(\sA)
\]
for the unique decomposition $D=D_1 D_2$ as a product of fundamental
discriminants such that $\abs{D_2}=\gcd(D,d)$, and where
\[
   V_0(t) := \begin{cases}
       0 & \text{if $t<0$} \\
       -\pi i & \text{if $t=0$} \\
       -2\pi i e^{-2\pi t} & \text{if $t>0$}
   \end{cases}
\]
\end{proposition}
\begin{proof}
The proof from \cite[(3.2) p.277]{GZ} works, with the following differences:
\begin{itemize}
    \item Replace $\varepsilon_{D_1}(N)$ by
        $\varepsilon_{D_1}(N\eta)$,
        as in \eqref{eq:sEs}.
    \item In the last step of the computation of $e_s(n,y)$, in \cite{GZ}
        they use the identity
        \[
        i\varepsilon_{D_1}(-N) = \frac{\varepsilon_D(N)}{i}\,
        \varepsilon_{D_2}(-N),
        \]
        which is only true in the case $(D_2,N)=1$. Thus a formula
        like theirs is good only so far as $\eta=1$, but ours is
        always true.
\end{itemize}
\end{proof}

Note that $\varepsilon_{\sA}(n,d)=0$ unless $\eta\mid d$. In
particular, $e_s(n,y)=0$ unless $\eta\mid n$.
To ease notation,
write $\eta^\ast=\kro{-1}{\eta}\eta$
and $D_0 = D/\eta^\ast$, so that $D = \eta^\ast D_0$ is a discriminant
decomposition, and $N$ is prime to $D_0$
(because we are assuming $D$ is odd, hence squarefree).

We let
\begin{equation}
\begin{split}
\tilde\varepsilon_{\sA}(n,d) & := 
\varepsilon_{\sA}(\eta n, \eta d) \\
&\phantom{:}=
\varepsilon_{D_1}(-Nd)\varepsilon_{\eta^\ast
D_2}(n/d)\chi_{D_1,\eta^\ast
D_2}(\sA),
\end{split}
\label{eq:tildevarepsilon}
\end{equation}
for the unique decomposition $D_0=D_1 D_2$ as a product of fundamental
discriminants such that $\abs{D_2}=\gcd(D_0,d)$.

\begin{corollary}
    \label{coro:fourier}
    The Fourier coefficients of
\[
G_{0,\sA}(z) = \frac{2\pi}{\sqrt{\abs{D}}} \sum_{m=0}^\infty b_{\sA}(m) q^m,
\]
are given by
\[
b_{\sA}(m) =
\sum_{n=0}^{{\abs{D}m}/{N}} \sigma_{\sA}(n)\,r_{\sA}(m\abs{D} - nN),
\]
where
\begin{align*}
\sigma_{\sA}(n)
& = \frac{\sqrt{\abs{D}}}{2\pi}\,e_0(\eta n,y)\, e^{2\pi \eta n y}
\\ &
= \begin{cases}
     \frac{\sqrt{\abs{D}}}{2\pi}L(\varepsilon_D,1) -
     \frac{\varepsilon_D(N\eta)}{2} \,
     L(\varepsilon_D,0)
     & \text{for $n=0$}
     \\
     \sum_{\substack{d\mid n\\d>0}} \tilde\varepsilon_{\sA}(n,d)
     & \text{for $n>0$}
 \end{cases}
\end{align*}
\end{corollary}
\begin{proof}
    Similar to \cite[(3.4) p.281]{GZ}.
\end{proof}

\begin{proposition}
\[
   \sigma_{\sA}(0) =
   \frac{1-\varepsilon_{D}(N\eta)}{2}\cdot \frac{h(D)}{u_D}
\]
\end{proposition}
\begin{proof}
\[
    \sigma_{\sA}(0) =
     \frac{\sqrt{\abs{D}}}{2\pi}L(\varepsilon_D,1) -
     \frac{\varepsilon_D(N\eta)}{2} \,
     L(\varepsilon_D,0).
\]
But
\[
\frac{\sqrt{\abs{D}}}{2\pi}L(\varepsilon_D,1) =
     L(\varepsilon_D,0) = \frac{h(D)}{u_D},
\]
by the class number formula and the functional equation for
$L(\varepsilon_D,s)$.
\end{proof}

Denote:
\[
\sigma_{\sA}(n_0,n_1) :=
     \sum_{\substack{d_0\mid n_0\\d_0>0}} \tilde\varepsilon_{\sA}(n_0 n_1,d_0).
\]
Recall that the number of integral ideals of norm $n$
in $\QQ(\sqrt{D})$ is
\[
    R_D(n) = \sum_{\substack{d\mid n\\d>0}} \varepsilon_D(d).
\]

\begin{lemma} For $n>0$
\begin{enumerate}
\item If $d_0d_1\mid n$ with $\gcd(d_1,D)=1$, then

\[ 
\tilde\varepsilon_{\sA}(n,d_0 d_1) =
\tilde\varepsilon_{\sA}(n,d_0) \,
\varepsilon_D(d_1)
.
\]

\item If $n=n_0 n_1$ with $\gcd(n_1,n_0D)=1$, then
\begin{align*}
\sigma_{\sA}(n)
& = \sigma_{\sA}(n_0,n_1) \, R_D(n_1)
.
\end{align*}

\item Let $n=n_0n_1$, where $n_0=\gcd(n,D^\infty)$. Then
    \[ \sigma_{\sA}(n) = \sigma_{\sA}(n_0,n_1) \, R_D(n) \]

\end{enumerate}

\label{lemma:6.8}
\end{lemma}

\begin{proof}\
\begin{enumerate}
\item \label{l:1}
Follows from the definition of $\tilde\varepsilon_{\sA}(n, d)$ in
\eqref{eq:tildevarepsilon}, because the hypothesis implies $\gcd(D_0,d_0d_1)=\gcd(D_0,d_0)$.
\item \label{l:2}
if $\gcd(n_1,n_0D)=1$ then
\begin{align*}
\sigma_{\sA}(n_0 n_1)
& = \sum_{d_0\mid n_0} \sum_{d_1\mid n_1} \tilde\varepsilon_{\sA}(n_0 n_1,d_0 d_1)
& \text{(since $\gcd(n_0,n_1)=1$)}
\\
& = \sum_{d_0\mid n_0} \sum_{d_1\mid n_1} \tilde\varepsilon_{\sA}(n_0 n_1,d_0)\,\varepsilon_D(d_1)
& \text{(by part (\ref{l:1}))}
\\
& = \sigma_{\sA}(n_0,n_1) 
\sum_{\substack{d_1\mid n_1\\d_1>0}} \varepsilon_D(d_1)
\\
& = \sigma_{\sA}(n_0,n_1) R_D(n_1)
.
\end{align*}

\item \label{lemma:6.8:3}
Since $n/n_1 = n_0 \mid D^\infty$, there is a unique ideal of norm
$n/n_1$, hence $R_D(n_1) = R_D(n)$. Thus the statement follows directly from
(\ref{l:2}).
\end{enumerate}
\end{proof}

\begin{lemma}
Let $\tilde\eta = \gcd(n, \eta^\infty)$, and $n'=\gcd(n, D_0^\infty)$.
    Then
    \[
    \sigma_{\sA}(n) = \tilde\varepsilon_{\sA}(n,\tilde\eta)\cdot
       \sum_{d'\parallel n'}
            \frac{\tilde\varepsilon_{\sA}(n,\tilde\eta d')}
                 {\tilde\varepsilon_{\sA}(n,\tilde\eta)}
       \cdot R_D(n)
   \]
   \label{lemma:6.9}
\end{lemma}
\begin{proof}
First note that 
    \[
    \tilde\varepsilon_{\sA}(n,\tilde\eta) = 
  \varepsilon_{D_0}(-N\,\tilde\eta)
  \, \varepsilon_{\eta^\ast}(n/\tilde\eta)
  \, \chi_{D_0,\eta^\ast}(\sA).
    \]
Since $\gcd(\eta,n/\tilde\eta)=1$, and $\gcd(D_0,N\tilde\eta)=1$,
it follows that $$\tilde\varepsilon_{\sA}(n,\tilde\eta)\neq 0.$$
\par
Let $n=n_0 n_1$ with $n_0=\gcd(n,D^\infty)$ as in the previous lemma,
and note that $n_0 = \tilde\eta n'$, since $\eta$ and $D_0$ are
relatively prime.
\par
Suppose $d_0\mid n_0$ is such that
$\tilde\varepsilon_{\sA}(n, d_0)\neq 0$.
By the definition of 
$\tilde\varepsilon_{\sA}$, it follows that
\[
  \varepsilon_{\eta^\ast}(n/d_0) \neq 0
  \qquad\text{and}\qquad
  \varepsilon_{D_2}(n/d_0) \neq 0,
\]
where $\abs{D_2}=\gcd(D_0, d_0)$.
The first inequality implies that
$\tilde\eta\mid d_0$, so we can write $d_0 = \tilde\eta d'$, where $d'\mid
n'$. Since $d'\mid n'\mid D_0^\infty$, and $\abs{D_2}=\gcd(D_0, d_0)$,
it follows that $d'\mid D_2^\infty$.
The second inequality implies $\varepsilon_{D_2}(n'/d')\neq 0$, 
and so we finally conclude that $d'\parallel n'$.

It follows from the above discussion that
    \begin{align*}
    \sigma_{\sA}(n_0,n_1) & =
       \sum_{d'\parallel n'}
            \tilde\varepsilon_{\sA}(n,\tilde\eta d')
       \\
       & =
       \tilde\varepsilon_{\sA}(n,\tilde\eta)\cdot
       \sum_{d'\parallel n'}
            \frac{\tilde\varepsilon_{\sA}(n,\tilde\eta d')}
                 {\tilde\varepsilon_{\sA}(n,\tilde\eta)}
    \end{align*}
    This finishes the proof by Lemma~\ref{lemma:6.8} (\ref{lemma:6.8:3}).
\end{proof}

\begin{lemma}
    The function
    \[
    d'\mapsto 
            \frac{\tilde\varepsilon_{\sA}(n,\tilde\eta d')}
                 {\tilde\varepsilon_{\sA}(n,\tilde\eta)}
                 \]
                 is multiplicative.
                 In particular, 
                 \[
       \sum_{d'\parallel n'}
            \frac{\tilde\varepsilon_{\sA}(n,\tilde\eta d')}
                 {\tilde\varepsilon_{\sA}(n,\tilde\eta)}
                 =\begin{cases}
                     \delta(n) & \text{if all terms are $1$,} \\
                     0 & \text{otherwise}
                 \end{cases}
                 \]
where $\delta(n):=2^t$, with $t$ the number of prime factors of
$\gcd(n,D_0)$.
\label{lemma:6.10}
\end{lemma}
\begin{proof}
    Let $D_0=D_1 D_2$ where $\abs{D_2} = \gcd(D_0,d')$.
    We have
    \[
    \tilde\varepsilon_{\sA}(n, \tilde\eta d') =
    \varepsilon_{D_1}(-N\tilde\eta d')\,
    \varepsilon_{\eta^\ast D_2}\left(\frac{n}{\tilde\eta d'}\right)
    \chi_{D_1,\eta^\ast D_2}(\sA)
    \]
    and
    \[
    \tilde\varepsilon_{\sA}(n,\tilde\eta) =
    \varepsilon_{D_0}(-N\tilde\eta)\,
    \varepsilon_{\eta^\ast}\left(\frac{n}{\tilde\eta}\right)
    \chi_{D_0,\eta^\ast}(\sA)
    \]
    Hence
    \begin{align*}
    \frac{\tilde\varepsilon_{\sA}(n,\tilde\eta d')}
         {\tilde\varepsilon_{\sA}(n,\tilde\eta)}
         & = \varepsilon_{D_1}(d') \, \varepsilon_{D_2}(-N\tilde\eta) \,
         \varepsilon_{\eta^\ast}(d') \, \varepsilon_{D_2}\left(\frac{n}{\tilde\eta d'}\right)
         \chi_{\eta^\ast D_1,D_2}(\sA)
         \\
    & =\varepsilon_{\eta^\ast D_1}(d')\,\varepsilon_{D_2}\left(-N\frac{n}{d'}\right)
    \chi_{\eta^\ast D_1,D_2}(\sA)
    \end{align*}
    (we have used $\chi_{D_1,\eta^\ast D_2} = \chi_{D_0,\eta^\ast}
    \cdot \chi_{\eta^\ast D_1, D_2}$).

    To check multiplicativity, let $d'=d''d'''$ with
    $\gcd(d'',d''')=1$, and let $D_0=D_1''D_2''=D_1'''D_2'''$ be the
    discriminant decompositions corresponding to $d''$ and $d'''$,
    i.e. $\abs{D_2''} =\gcd(D_0,d'')$ and $\abs{D_2'''}=\gcd(D_0,d''')$.
    Note that $D_2=D_2'' D_2'''$, and so $D_1'' = D_1 D_2'''$ and
    $D_1''' = D_1 D_2''$.

    Then
    \[
    \frac{\tilde\varepsilon_{\sA}(n,\tilde\eta d'')}
         {\tilde\varepsilon_{\sA}(n,\tilde\eta)}
         =\varepsilon_{\eta^\ast
         D_1''}(d'')\varepsilon_{D_2''}(d''')\varepsilon_{D_2''}\left(-N\frac{n}{d'}\right)
         \chi_{\eta^\ast D_1'',D_2''}(\sA)
    \]
    and
    \[
    \frac{\tilde\varepsilon_{\sA}(n,\tilde\eta d''')}
         {\tilde\varepsilon_{\sA}(n,\tilde\eta)}
         =\varepsilon_{\eta^\ast
         D_1'''}(d''')\varepsilon_{D_2'''}(d'')\varepsilon_{D_2'''}\left(-N\frac{n}{d'}\right)
         \chi_{\eta^\ast D_1''',D_2'''}(\sA).
    \]
    Hence the product
    \begin{multline*}
    \frac{\tilde\varepsilon_{\sA}(n,\tilde\eta d'')}
         {\tilde\varepsilon_{\sA}(n,\tilde\eta)}
         \cdot
    \frac{\tilde\varepsilon_{\sA}(n,\tilde\eta d''')}
         {\tilde\varepsilon_{\sA}(n,\tilde\eta)} 
         \\ 
    \begin{aligned}
         & = \varepsilon_{\eta^\ast D_1''D_2'''}(d'')
           \varepsilon_{\eta^\ast D_1'''D_2''}(d''')
           \varepsilon_{D_2}\left(-N\frac{n}{d'}\right)
           \chi_{\eta^\ast D_1,D_2}(\sA)
           \\
         & = \varepsilon_{\eta^\ast D_1 (D_2''')^2}(d'')
           \varepsilon_{\eta^\ast D_1 (D_2'')^2}(d''')
           \varepsilon_{D_2}\left(-N\frac{n}{d'}\right)
           \chi_{\eta^\ast D_1,D_2}(\sA)
           \\
         & = \varepsilon_{\eta^\ast D_1}(d')
           \varepsilon_{D_2}\left(-N\frac{n}{d'}\right)
           \chi_{\eta^\ast D_1,D_2}(\sA)
           \\
         & = \frac{\tilde\varepsilon_{\sA}(n,\tilde\eta d')}
                  {\tilde\varepsilon_{\sA}(n,\tilde\eta)}.
     \end{aligned}
     \end{multline*}
\end{proof}

\begin{lemma}Suppose there is an ideal $\id{a}\in\sA$ such that
    $\normid{a} \equiv -nN\pmod{D}$, and let
            $\id{b}$ be an ideal of norm $n$.
            Then the following
            conditions are equivalent:
    \begin{enumerate}
     \item \label{x:1} $\frac{\tilde\varepsilon_{\sA}(n,\tilde\eta d')}
         {\tilde\varepsilon_{\sA}(n,\tilde\eta)}=1$ for all $d'\parallel n'$.
     \item \label{x:2}
         $\chi_{l^\ast,D/l^\ast}(\id{ab}) =
         \varepsilon_{l^\ast}(-N)$ 
         for all prime discriminants $l^\ast \mid D_0$.
     \item \label{x:3}
	 There is an ideal $\id{q}$ in the same genus as $\id{ab}$
         such that \[\normid{q}\equiv -N\pmod{D_0}.\]
    \end{enumerate}
    Moreover, this also implies
    \[
    \tilde\varepsilon_{\sA}(n,\tilde\eta) = 1
    \]
    \label{lemma:6.11}
\end{lemma}
\begin{proof}
	We first prove that (\ref{x:1}) implies (\ref{x:2}).
    Let $l^\ast$ be a prime discriminant with $l\mid D_0$, and consider the following two cases:
    \begin{itemize}
        \item $l\nmid n$: then $\gcd(\normid{ab},l)=1$, so
            $\chi_{l^\ast,D/l^\ast}(\id{ab})=\varepsilon_{l^\ast}(\normid{ab})=\varepsilon_{l^\ast}(-N)$,
            by hypothesis.
    \item $l\mid n$: use (\ref{x:1}) with $d'=\gcd(n,l^\infty)$, thus
            $|D_2|=l$, and we have
            \begin{align*}
	    \chi_{D/l^\ast,l^\ast}(\id{a})
	    & = \varepsilon_{D/l^\ast}(d')\, \varepsilon_{l^\ast}(-N\frac{n}{d'})
	    \\
	    & = \varepsilon_{l^\ast}(-N)\, \varepsilon_{D/l^\ast}(d')\, \varepsilon_{l^\ast}(n/d')
	    \\
	    & = \varepsilon_{l^\ast}(-N)\, \chi_{D/l^\ast,l^\ast}(\id{b}),
	    \end{align*}
            where the last equality holds since $d'\mid D$.
    \end{itemize}

    To prove that (\ref{x:2}) implies (\ref{x:1}), note that
    since the expression in (\ref{x:1}) is multiplicative, it is enough to
    check it for $d'=\gcd(n,l^\infty)$, where $l$ is any prime
    dividing $n'$. 
    Since $l\mid n$, a computation similar to the second case above applies. 

    Clearly (\ref{x:3}) implies (\ref{x:2}); to see the converse
    take an ideal $\id{c}$ in the same genus as
    $\id{ab}$, with $\gcd(\normid{c},D_0)=1$. By (\ref{x:2}), we know that
    \[
       \normid{c} \equiv -N r^2 \pmod{D_0}
    \]
    for some $r\in(\ZZ/D_0)^\times$, and we can take
    $\id{q}=\tilde{r}\id{c}$ where $\tilde{r}\in\ZZ$ is such that
    $\tilde{r}r\equiv 1\pmod{D_0}$.

    For the final assertion, we use the definition of
    $\tilde\varepsilon_{\sA}$ in \eqref{eq:tildevarepsilon} with
    $d=\tilde\eta$, so that
    $D_1=D_0$ and $D_2=\eta^\ast$, and thus
    \begin{align*}
    \tilde\varepsilon_{\sA}(n,\tilde\eta)
    & = \varepsilon_{D_0}(-N\tilde\eta)\,
    \varepsilon_{\eta^\ast}(n/\tilde\eta)\,
    \chi_{D_0,\eta^\ast}(\sA)
    \\
    & = \varepsilon_{D_0}(-N)\,\chi_{D_0,\eta^\ast}(\id{b})\,
    \chi_{D_0,\eta^\ast}(\sA),
    \end{align*}
    since
    $\chi_{D_0,\eta^\ast}(\id{b})=\varepsilon_{D_0}(\tilde\eta)\,
                                  \varepsilon_{\eta^\ast}(n/\tilde\eta)$.
    Hence,
    \begin{align*}
    \tilde\varepsilon_{\sA}(n,\tilde \eta)
    & = \varepsilon_{D_0}(-N)
    \,
    \chi_{D_0,\eta^\ast}(\id{q})
    \\
    & = \varepsilon_{D_0}(-N)
    \,
    \varepsilon_{D_0}(-N) = 1.
    \end{align*}
\end{proof}

Denote by $R_{\gen{b}}(n)$ the number of integral ideals of norm $n$
in a given genus $\gen{b}$.
We can finally obtain a closed formula for $\sigma_{\sA}(n)$ when $n>0$:
\begin{proposition}\label{prop:sigmaA}
    For $n>0$, 
    suppose
    there is an ideal $\id{a}\in\sA$ such that
    $\normid{a} \equiv -nN\pmod{D}$.
    Then
\[
\sigma_{\sA}(n)
= \delta(n) \sum_{\gen{q}\in Q}
R_{\gen{\sA\id{q}}}(n)
\]
where the sum is over the set of genera
\[
Q := \set{ \gen{q} \st \normid{q}\equiv -N\pmod{D_0}}.
\]
\end{proposition}
\begin{proof}
We have
\begin{align*}
        \sigma_{\sA}(n)
        &= \tilde\varepsilon_{\sA}(n,\tilde \eta) \cdot \sum_{d'\parallel
        n'}\frac{\tilde\varepsilon_{\sA}(n,\tilde \eta
        d')}{\tilde\varepsilon_{\sA}(n,\tilde \eta)} \cdot R_D(n)
           & \text{(by Lemma \ref{lemma:6.9})}
        \\
        &= \tilde\varepsilon_{\sA}(n,\tilde \eta) \cdot \sum_{d'\parallel
        n'}\frac{\tilde\varepsilon_{\sA}(n,\tilde \eta
        d')}{\tilde\varepsilon_{\sA}(n,\tilde \eta)}
        \cdot \sum_{\substack{\id{b}\\\normid{b}=n}} 1
        \\
        &=\sum_{\substack{\id{b}\\\normid{b}=n}}
        1
        \cdot
        \begin{cases}
            \delta(n) & \text{if $\gen{ab}\in Q$,} \\
            0 & \text{otherwise.}
        \end{cases}
\end{align*}
where the last equality follows from Lemma~\ref{lemma:6.10}
and Lemma~\ref{lemma:6.11}.
\par
Now we note that $\gen{ab}\in Q$ is equivalent to $\gen{b}=\gen{aq}$ for
some $\gen{q}\in Q$, hence we can rewrite the last summation as
\begin{align*}
\sigma_{\sA}(n)
& = 
\delta(n)
\sum_{\gen{q}\in Q}
\sum_{\substack{\id{b}\\\normid{b}=n\\\gen{b}=\gen{aq}}}
1
\\
& =
\delta(n) \sum_{\gen{q}\in Q} R_{\gen{aq}}(n)
\end{align*}
\end{proof}

\begin{remark}\
    \begin{enumerate}
        \item In the case $\eta=1$, we have $D_0=D$, and the condition
            on $\normid{q}$ in the definition of $Q$ determines its
            genus $\gen{q}$, in case it exists. This depends on the
            sign of $\varepsilon_D(\normid{q})$, so we have
            \[
            \#Q = \begin{cases}
                1 & \text{if $\varepsilon_D(N)=-1$,} \\
                0 & \text{if $\varepsilon_D(N)=1$.}
            \end{cases}
            \]
            This is the case of \cite{GZ}, and in this case the proposition
            is part (a) of Proposition 4.6 in \cite[p.285]{GZ}.
            As noted before, when $N$ is a perfect square
            we have $\#Q=0$ for all $D$ prime to $N$.
        \item When $\eta\neq 1$, we have $Q\neq\emptyset$. Indeed,
            for any $\alpha\in\ZZ$ such that
            $\varepsilon_{\eta^\ast}(\alpha)=\varepsilon_{D_0}(-N)$,
            by genus theory there is an ideal $\id{q}$ with
    \[
    \normid{q}\equiv
    \begin{cases}
    -N\pmod{D_0}, \\
    \alpha \pmod{\eta}.
    \end{cases}
    \]
    The number of such $\alpha \bmod \eta$, up to squares is
    $2^{t-1}$, where $t$ is the number of prime factors of $\eta$.
    Each one results in an ideal lying in a different genus,
    hence
    \[
    \#Q = 2^{t-1}.
    \]
\item In the particular case $N=p^2$, it follows that
    \[
    \#Q =
    \begin{cases}
        0 & \text{if $p\nmid D$,} \\
        1 & \text{if $p\mid D$.}
    \end{cases}
    \]
    \end{enumerate}
\end{remark}

\subsection{The central value of $L$-series}

We conclude this section with a formula for the central value of
$L$-series which is similar to \cite[(4.4) p.283]{GZ}, but not requiring
$\gcd(D,N)=1$ as in the original formulation.
As remarked above, this generalization is essential to obtain a
non-trivial result in the case of level $p^2$, which is the main
interest of this paper.

\begin{theorem}
\label{thm:rankin}
Let $D<0$ be an odd fundamental discriminant, $\sA$ be an ideal in
$\Q[\sqrt{D}]$ and $f(z)$ be a cusp form in
$S_2^\new(\Gamma_0(N))$. Then, 
    \[
    L_{\sA}(f,1) = \frac{8\pi^2}{\sqrt{\abs{D}}} \<f, g_{\sA}>,
    \]
with $g_{\sA} = g_{\sA}^{(N)} = \sum b_{\sA}(m) q^m$,
where
\begin{multline*}
  b_{\sA}(m) :=
  \frac{1-\varepsilon_D(N\eta)}{2}\cdot \frac{h(D)}{u_D} r_{\sA}(m)
  \\
  +
  \sum_{\gen{q}\in Q}
  \sum_{n=1}^{{\abs{D}m}/{N}}
  \delta(n) r_{\sA}(m\abs{D} - nN)
  R_{\gen{\sA\id{q}}}(n),
\end{multline*}
where the first sum is over the set of genera
\[
Q := \set{ \gen{q} \st \normid{q}\equiv -N\pmod{D_0}},
\]
and where $\delta(n):=2^t$, with $t$ the number of prime factors of
$\gcd(n,D_0)$.
\end{theorem}
\begin{proof}
    This follows by combining
    Proposition~\ref{prop:rankin},
    Corollary~\ref{coro:fourier}, and
    Proposition~\ref{prop:sigmaA}, using the renormalization
    \[
    g_{\sA}(z) = \frac{\sqrt{\abs{D}}}{2\pi}\, G_{0,\sA}(z).
    \]
    Note that when evaluating each term
    $\sigma_{\sA}(n)\,r_{\sA}(m\abs{D} - nN)$ in the sum of
    Corollary~\ref{coro:fourier}, the hypothesis of
    Proposition~\ref{prop:sigmaA} holds whenever
    $r_{\sA}(m\abs{D} - nN)\neq 0$, so we can indeed substitute the
    value of $\sigma_{\sA}(n)$ without restriction.
\end{proof}

\begin{remark}
We expect a similar formula to hold for any fundamental discriminant
$D<0$. The case of even discriminants is harder since the discriminant
is not square free in this case. Nevertheless, the result for odd
discriminants will be enough for our purposes.
\end{remark}

\def\MR#1{}
\bibliography{bibliography}
\bibliographystyle{abbrv}

\end{document}